\documentclass{article}

\usepackage[utf8]{inputenc}
\usepackage[T1]{fontenc}
\usepackage{hyperref}
\usepackage{url}
\usepackage{booktabs}
\usepackage{amsfonts}
\usepackage{nicefrac}
\usepackage{microtype}
\usepackage{lipsum}
\usepackage{graphicx}
\usepackage{amsmath}
\usepackage{amssymb}
\usepackage{amsthm}
\usepackage{amscd}
\usepackage{alltt}
\usepackage[all]{xy}

\newtheorem{lemma}{Lemma}
\newtheorem{prop}{Proposition}
\newtheorem{theorem}{Theorem}
\newtheorem*{theorem*}{Theorem}

\theoremstyle{definition}
\newtheorem{example}{Example}

\newcommand{\rar}{\rightarrow}
\newcommand{\sMat}[4]{\bigl(\begin{smallmatrix}#1\mathstrut&#2\mathstrut\\#3\mathstrut&#4\mathstrut\end{smallmatrix}\bigr)}

\newcommand{\dotminus}{\mathbin{\dot{-}}}

\DeclareMathOperator\mat{M}
\DeclareMathOperator\gl{GL}
\DeclareMathOperator\unit{U}
\DeclareMathOperator\eunit{EU}
\DeclareMathOperator\gunit{GU}
\DeclareMathOperator\punit{PU}
\newcommand{\kunit}{\mathrm{KU}_1}
\DeclareMathOperator\Cent{C}

\DeclareMathOperator\tr{tr}
\DeclareMathOperator\Norm{N}
\DeclareMathOperator\Ker{Ker}
\DeclareMathOperator\herm{H}

\newcommand{\eps}{\varepsilon}
\newcommand{\leqt}{\trianglelefteq}
\newcommand{\inv}[1]{\!\;\overline{\!\!\:#1\vphantom !\!\!\:}\;\!}
\newcommand{\op}{{\mathrm{op}}}
\DeclareMathOperator{\id}{id}
\newcommand{\rrar}{\Rightarrow}
\DeclareMathOperator{\Hom}{Hom}
\DeclareMathOperator{\End}{End}
\DeclareMathOperator{\Aut}{Aut}
\DeclareMathOperator{\Heis}{Heis}

\newcommand{\sub}[2]{{_{#1\!}{#2}}}
\newcommand{\up}[2]{{^{#1}\!{#2}}}
\newcommand{\blank}{{}}

\makeatletter
\newcommand{\bigperp}{\mathop{\mathpalette\bigp@rp\relax}\displaylimits}
\newcommand{\bigp@rp}[2]{\vcenter{\m@th\hbox{\scalebox{\ifx#1\displaystyle2.1\else1.5\fi}{\(#1\perp\)}}}}
\makeatother

\newcommand{\rMod}[2][]{\mathbf{Mod}_{#1}\text{-}#2}

\begin{document}

\title{Stability for odd unitary \(\mathrm K_1\)}


\author{
  Egor Voronetsky
  \thanks{Research is supported by the Russian Science Foundation grant 19-71-30002.} \\
  \small Chebyshev Laboratory, \\
  \small St. Petersburg State University, \\
  \small 14th Line V.O., 29B, \\
  \small Saint Petersburg 199178 Russia \\
  \small \texttt voronetckiiegor@yandex.ru
}

\maketitle

\begin{abstract}
We give a new purely algebraic approach to odd unitary groups using odd form rings. Using these objects, we prove the stability theorems for odd unitary \(K_1\)-functor without using the corresponding result from linear \(K\)-theory under the ordinary stable rank condition. Moreover, we give a natural stabilization result for projective unitary groups and various general unitary groups.
\end{abstract}

\section{Introduction}

The modern definition of odd unitary groups was given in \cite{OddDefPetrov} by Victor Petrov. His definition generalizes Anthony Bak's unitary groups and split odd orthogonal groups, hence all classical Chevalley groups over arbitrary commutative rings. In \cite{OddStrucVor} we introduced quadratic structures that may be used to construct odd unitary groups in the same way as Bak's unitary groups arise from so-called form parameters.

The natural problem in this context is to prove results from the unitary \(K\)-theory. For general linear groups this was done in paper \cite{LinStabBass} by Hyman Bass and also may be found in his book \cite{Bass}. For classical and Bak' unitary groups the stability was proved in \cite{UnitStabBak1, UnitStabBak2, UnitStabVaser, ClassicStabVaser} and in unpublished paper \cite{UnitStabSaliani}, the last one may be found on Max-Albert Knus's homepage. Finally, for Petrov's unitary groups the surjectivity was already proved in \cite{OddDefPetrov}, and the injectivity in the main result of Yu Weibo's paper \cite{OddStabWeibo}. All these injectivity proofs for various unitary groups used the corresponding result for linear groups.

For unitary groups the stabilization is usually formulated in terms of the \(\Lambda\)-stable rank condition, starting from \cite{UnitStabBak1}. This condition is weaker than all the previous ones, in particular, \(\Lambda\)-stable rank may be bound by the absolute stable rank. For algebras \(R\) over a commutative ring \(K\) the absolute stable rank may be bound by the Bass --- Serre dimension of \(K\) (more precisely, \(\mathrm{asr}(R) \leq \mathrm{BS}(K) + 1\)).

On the other hand, in some cases the injectivity was proved using only the stable rank of \(R\). For Chevalley groups it was done in Stein's paper \cite{ChevStabStein}, and for even unitary groups it is the main result of S. Sinchuk's work \cite{UnitStabSin}.

In his paper \cite{LinNilpBak} Bak proved the nilpotence of linear \(K_1\)-groups using his localization-completion method. This result is much stronger that the stabilization and may be proved independently. The nilpotence for Bak's unitary groups was proved in the absolute case by Hazrat in \cite{UnitNilpHazrat} and in the relative case by Bak, Hazrat, and Vavilov in \cite{UnitNilpVav}. For Petrov's unitary groups it was done in recent Weibo and Tang's paper \cite{OddNilpWeibo}.

In paper \cite{OddStrucVor} we already implicitly used odd form rings in the definition of levels. A (unital special) odd form ring is a pair \((R, \Delta)\), were \(R\) is an involution ring and \(\Delta\) is Petrov's odd form parameter on the right \(R\)-module \(R\) with the canonical hermitian form. As was shown in that paper, unitary groups of arbitary regular quadratic bimodules may be considered as unitary groups of appropriate odd form rings. We will prove that the regularity condition is redundant.

In order to prove relative results for various groups, the common approach is Stein's relativization technique from \cite{Stein}. However, even for Bak's unitary groups it becomes complicated because the definition of Bak's form rings is not truly algebraic (from the point of view of universal algebra). This difficulty leads to the relativization with two parameters, see \cite{UnitElemVav} for Bak's unitary groups and \cite{OddDefPetrov} for Petrov's unitary groups. Our odd form rings will be defined using operations and axioms, hence the original Stein's variant may be applied.

In section \(2\) we will review odd quadratic modules from \cite{OddStrucVor}, and in sections \(3\)--\(7\) we will develop general theory of odd form rings. Section \(8\) contains all examples, it uses our main results from the last sections. Sections \(9\) and \(10\) contains definition and properties of the \(\Lambda\)-stable rank in our situation, as well as a direct proof that this value may be bounded by the Bass --- Serre dimension of the underlying commutative ring. The main result of section \(11\) is
\begin{theorem*}
Let \((R, \Delta)\) be an odd form ring with a free orthogonal family of rank \(n\). Then the map
\[\kunit(n - 1; R, \Delta) \rar \kunit(n; R, \Delta)\]
is surjective for \(\Lambda\mathrm{sr}(\eta_1; R, \Delta) \leq n - 1\) and injective for \(\mathrm{sr}(\eta_1; R, \Delta) \leq n - 2\). Also, if \((I, \Gamma) \leqt (R, \Delta)\) is an odd form ideal, then the map
\[\kunit(n - 1; R, \Delta; I, \Gamma) \rar \kunit(n; R, \Delta; I, \Gamma)\]
is surjective for \(\Lambda\mathrm{sr}(\eta_1; I \rtimes R, \Gamma \rtimes \Delta) \leq n - 1\) and injective for \(\mathrm{sr}(\eta_1; I \rtimes R, \Gamma \rtimes \Delta) \leq n - 2\).
\end{theorem*}
The proof does not use the stability for linear groups and for these groups the theorem reduces to the classical result of Bass. Finally, from the last section we have the following
\begin{theorem*}
Let \((R, \Delta)\) be an odd form ring with a free orthogonal hyperbolic family of rank \(n \geq 1\) and \((I, \Gamma) \leqt (R, \Delta)\) be an odd form ideal. Then there is canonical isomorphism
\[\punit^*(n - 1; R, \Delta) / \unit(n; R, \Delta) \cong \punit^*(n; R, \Delta) / \unit(n; R, \Delta)\]
for \(\Lambda\mathrm{sr}(\eta_1; R, \Delta) \leq n - 2\)
and there is canonical isomorphism
\[\punit^*(n - 1; R, \Delta; I, \Gamma) / \unit(n - 1; I, \Gamma) \cong \punit^*(n; R, \Delta; I, \Gamma) / \unit(n; I, \Gamma)\]
for \(\Lambda\mathrm{sr}(\eta_1; I \rtimes R, \Gamma \rtimes \Delta) \leq n - 2\). In all classical cases over a local commutative ring or a PID
\[\punit^*(n; R, \Delta) = \punit(n; R, \Delta).\]
\end{theorem*}

\section{Quadratic and hermitian forms}

Every ring in this paper is associative but not necessarily with \(1\). All commutative rings have identity elements and trivial involutions, homomorphisms between them are unital. When we work with algebras over a fixed commutative ring \(K\), we always consider only those bimodules that have the same \(K\)-module structure from the left and from the right. Recall that if \(R\) is a non-unital \(K\)-algebra, then \(R \rtimes K\) is a unital \(K\)-algebra with an ideal \(R\), it is the \(K\)-module \(R \oplus K\) the multiplication on \(R \rtimes K\) is given by \((r, k1) (r', k'1) = (rr' + rk' + r'k, kk'1)\).

If \(R\) is arbitrary ring, then \(R^\bullet\) is the multiplicative semigroup of \(R\) (it is a monoid, if \(R\) is unital), i.e. the set \(R\) with the multiplication operation. By \(\Cent(G)\) and \(\Cent(R)\) we denote the center of a group \(G\) or of a ring \(R\). The subgroups \(\Cent_G(H)\) and \(\Norm_G(H)\) are the centralizer and the normalizer of a subgroup \(H\) in a group \(G\). We use the notations \(\up g h = ghg^{-1}\) and \([g, h] = ghg^{-1}h^{-1}\) for elements \(g\) and \(h\) of arbitrary group. Also, if \(f \colon A \rar B\) is a homomorphism of algebraic structures, then we sometimes will write \(\up f a\) instead of \(f(a)\). If a group \(G\) acts on a set \(X\) and \(x \in X\), then \(G_x = \{g \in G \mid gx = x\}\) is the stabilizer of \(x\).

If \(K\) is a commutative ring and \(\mathfrak p \leqt K\) is a prime ideal, then \(\kappa(\mathfrak p) = K_{\mathfrak p} / \mathfrak p K_{\mathfrak p}\) is the residue field of \(\mathfrak p\). The sets \(\mathrm{Spec}(K)\) and \(\mathrm{Max}(K)\) of all prime and maximal ideals of \(K\) are equipped with the Zariski topology. If \(R\) is arbitrary ring and \(n \geq 0\) is an integer, then \(\mat(n, R)\) is the ring of \((n \times n)\)-matrices with entries in \(R\).

If \(R\) is a ring, then \(R^\op = \{r^\op \mid r \in R\}\) is the opposite ring, \(r^\op (r')^\op = (r'r)^\op\). The same notation is used for the opposite modules (note that the opposite to a left module is a right modules and vice versa). Clearly, \((R^\op)^\op \cong R\), and similarly to modules. An involution on a ring \(R\) is an additive map \(\inv{(-)} \colon R \rar R\) such that \(\inv{rr'} = \inv{r'} \inv {r\vphantom{r'}}\), and \(\inv{\inv r} = r\) (if \(R\) is unital it follows that \(\inv 1 = 1\)). Finally, \(\herm(R) = \{r \in R \mid \inv r = r\}\) is the set of hermitian elements of a fixed involution on \(R\).

Until the end of this section all rings, ring homomorphisms, and modules are unital.

Recall the definitions from \cite{OddStrucVor}. Let \(R\) be arbitrary ring and \(\lambda \in R^*\). A map \(\inv{(-)} \colon R \rar R, r \mapsto \inv r\) is called a \(\lambda\)-involution (or a pseudo-involution), if it is additive, \(\inv 1 = 1\), \(\inv{rr'} = \inv{r'} \inv {r\vphantom{r'}}\), \(\inv{\inv r} = \lambda r \lambda^{-1}\), and \(\inv \lambda = \lambda^{-1}\). For example, if \(R\) is commutative, then the identity map on \(R\) is a \(1\)-involution (i.e. an involution) and a \((-1)\)-involution simultaneously. If \(R = S^\op \times S\), then \(\inv{(a^\op, b)} = (b^\op, a)\) is also an involution. Let \(M_R\) be a right module. A map \(B \colon M \times M \rar R\) is called a hermitian form, if it is biadditive, \(B(m, m'r) = B(m, m')r\), and \(B(m', m) = \inv{B(m, m')} \lambda\). The hermitian form \(B\) is called regular (or non-degenerate), if \(M_R\) is finitely generated projective and \(B\) induces an isomorphism \(M^\op \cong \Hom_R(M, R), m^\op \mapsto B(m, -)\). The module \(R_R\) has the canonical regular hermitian form \(B_1(r, r') = \inv rr'\).

A quadratic structure on a ring \(R\) with a \(\lambda\)-involution is a right \(R^\bullet\)-module \(A\) (i.e. an abelian group with a right action of monoid \(R^\bullet\)) with additive maps \(\varphi \colon R \rar A\) and \(\tr \colon A \rar R\) such that
\begin{align*}
\varphi(\inv rr' r) &= \varphi(r') \cdot r,\\
\tr(a \cdot r) &= \inv r \tr(a) r,\\
\tr(\varphi(r)) &= r + \inv r \lambda,\\
\tr(a) &= \inv{\tr(a)} \lambda,\\
\varphi(r) &= \varphi(\inv r \lambda),\\
a \cdot (r + r') &= a \cdot r + \varphi(\inv r' \tr(a) r) + a \cdot r'.
\end{align*}

For example, if \(\Lambda \leq R\) is a form parameter (that is, \(\{r - \inv r \lambda\} \leq \Lambda \leq \{r \mid r + \inv r \lambda = 0\}\) and \(\inv r \Lambda r \leq \Lambda\) for all \(r\)), then \(A = R / \Lambda\) is a quadratic structure with \(\varphi(r) = r + \Lambda\) and \(\tr(r + \Lambda) = r + \inv r \lambda\).

Let \((M_R, B)\) be a hermitian module and \(A\) be a quadratic structure on \(R\). A map \(q \colon M \rar A\) is called a quadratic form, if
\begin{align*}
q(mr) &= q(m) \cdot r,\\
\tr(q(m)) &= B(m, m),\\
q(m + m') &= q(m) + \varphi(B(m', m)) + q(m').
\end{align*}
In the case of Bak's quadratic forms, \(A = R / \Lambda\) for a form parameter \(\Lambda\), \(B(m, m') = Q(m, m') + \inv{Q(m', m)} \lambda\) for a sesquilinear map \(Q\), and \(q(m) = \varphi(Q(m, m))\). The unitary group of a quadratic module \((M, B, q)\) is \(\unit(M, B, q) = \{g \in \Aut_R(M) \mid B(gm, gm') = B(m, m'), q(gm) = q(m) \text{ for all } m, m'\}\).

As was shown in \cite{OddStrucVor}, Petrov's quadratic forms are almost the same as our quadratic forms if \(A\) is generated by the images of \(\varphi\) and \(q\). Indeed, recall that the Heisenberg group of \((M, B)\) is \(\Heis(M, B) = M \times R\) with the operation \((m, r) \dotplus (m', r') = (m + m', r - B(m, m') + r')\) (the identity element is \(\dot 0 = (0, 0)\) and the inverses are \(\dotminus (m, r) = (-m, -B(m, m) - r)\)). The monoid \(R^\bullet\) acts on the group \(\Heis(M, B)\) from the right via \((m, r) \cdot r' = (mr', \inv{r'} rr')\), and there are natural maps \(\varphi \colon R \rar \Heis(M, B), r \mapsto (0, r)\), \(\tr \colon \Heis(M, B) \rar R, (m, r) \mapsto B(m, m) + r + \inv r \lambda\), and \(q \colon M \rar \Heis(M, B), m \mapsto (m, 0)\). A Petrov's odd form parameter is an \(R^\bullet\)-subgroup \(\mathcal L \leq \Heis(M, B)\) such that \(\{(0, r - \inv r \lambda)\} \leq \mathcal L \leq \Ker(\tr)\). Then \(\phi\), \(\tr\), and \(q\) are well-defined on \(A = \Heis(M, B) / \mathcal L\). Conversely, if a quadratic structure \(A\) with a fixed quadratic form \(q\) is generated by the images of \(\varphi\) and \(q\), then \(A\) is isomorphic to \(\Heis(M, B) / \mathcal L\) for a unique \(\mathcal L\).

Recall also that if \((P, B, q)\) and \((P', B', q')\) are quadratic modules over a ring \(R\) with \(\lambda\)-involution and quadratic structure \(A\), then \((P, B, q) \perp (P', B', q') = (P \oplus P', B \perp B', q \perp q')\) is also a quadratic modules, where
\[(B \perp B')(p_1 \oplus p'_1, p_2 \oplus p'_2) = B(p_1, p_2) + B'(p'_1, p'_2)\]
and
\[(q \perp q')(p \oplus p') = q(p) + q'(p').\]
Conversely, if a quadratic module \((P, B, q)\) splits as a direct sum of orthogonal submodules \(P = P_1 \oplus P_2\) (that is, \(B(P_1, P_2) = 0\)), then \(P = P_1 \perp P_2\).

Let \((P, B_P, q_P)\) be a quadratic module, \(P \in \rMod R\), \(A\) be the quadratic structure over \(R\). Suppose that \(P_R\) is finitely generated projective. Then \(Q = \Hom_R(P, R)^\op\) is also an \(R\)-module and the map \(\langle -, = \rangle \colon Q \times P \rar R, (x^\op, p) \mapsto x(p)\) is the canonical coupling. Now the module \(\mathrm M(P) = Q \oplus P\) has a hermitian form \(B(x^\op \oplus p, x'{}^\op \oplus p') = B_P(p, p') + \langle x^\op, p' \rangle + \langle x'{}^\op, p \rangle\) and a quadratic form \(q(x^\op \oplus p) = q_P(p) + \varphi(\langle x^\op, p \rangle)\). The module \(\mathrm M(P)\) is called a metabolic space with the lagrangian \(Q\) (clearly, the hermitian form of \(\mathrm M(P)\) is regular). Conversely, if \(M\) is a quadratic space and there is a direct summand \(Q \leq M\) such that \(B|_{Q \times Q} = 0\), \(q|_Q = 0\), and \(B|_{Q \times M / Q}\) induces an isomorphism \(M / Q \cong \Hom_R(Q, R)^\op\), then \(M\) is a metabolic space (constructed via the direct complement to \(Q\)). If \(B_P = 0\) and \(q_P = 0\), then \(\mathrm M(P)\) is called a hyperbolic space and is denoted by \(\mathrm H(P)\).

Note that if \((P, B|_{P \times P}, q|_P) \leq (M, B, q)\) is a regular submodule of a quadratic module, then \(M\) canonically splits as an orthogonal sum of \(P\) and its orthogonal complement \(P^\perp = \{m \in M \mid B(m, P) = 0\}\). We will usually consider the situation when a given quadratic module \((M, B, q)\) has several pairwise orthogonal hyperbolic submodules \(M_i\), \(1 \leq i \leq n\), with the lagrangians \(P_i\) and \(P_{-i}\) (i.e. \(M_i = P_i \oplus P_{-i}\), \(B|_{P_i \times P_i} = 0\), \(B|_{P_{-i} \times P_{-i}} = 0\), and \(B\) induces an isomorphism \(P_{-i} \cong \Hom_R(P_i, R)^\op\)). The orthogonal complement to all \(M_i\) will be denoted by \(M_0\). If \(P_i = R\) for \(0 < |i| \leq n\), then \(M\) is exactly an odd hyperbolic space of rank \(n\) from \cite{OddDefPetrov}.

\section{Odd form rings}

Let \((M, B, q)\) be a quadratic module over a unital ring \(R\) with a \(\lambda\)-involution and a quadratic structure \(A\). Consider the ring \(T = \{(x^\op, y) \in \End_R(M)^\op \times \End_R(M) \mid B(xm, m') = B(m, ym') \text{ for all } m, m'\}\), then clearly \(T\) possess an involution \(\inv{(x^\op, y)} = (y^\op, x)\). If \(B\) is regular, then \(T \cong \End_R(M), (x^\op, y) \mapsto y\). Note that \(\unit(T, B_1) = \{a \in T^* \mid a^{-1} = \inv a\} \cong \{g \in \Aut_R(M) \mid B(gm, gm') = B(m, m')\}\). Next, there is the odd form parameter
\begin{align*}
\Xi = \{&((x^\op, y), (z^\op, w)) \mid \\
&q(ym) + \varphi(B(m, wm)) = 0 \text{ for all } m, xy + z + w = 0\} \leq \Heis(T, B_1).
\end{align*}
It is easy to see that \(\unit(M, B, q) \cong \unit(T, B_1, q_\Xi)\), where \(q_\Xi(a) = (1, 0) \cdot a \dotplus \Xi \in \Heis(T, B_1) / \Xi\).

Conversely, let \(T\) be a unital involution ring and \(\Xi \leq \Heis(T, B_1)\) be an odd form parameter, \(\unit(T, B_1, q_\Xi)\) be the unitary group of \((T, B_1, q_\Xi)\). We will call the pair \((T, \Xi)\) a special unital odd form ring. Clearly, there are natural maps \(\pi \colon \Xi \rar T, (a, b) \mapsto a\), \(\rho \colon \Xi \rar T, (a, b) \mapsto b\), and \(\phi \colon T \rar \Xi, a \mapsto \varphi(a - \inv a) = (0, a - \inv a)\). For any \(g \in \unit(T, B_1)\) (i.e. if \(g \in T^*\) and \(g^{-1} = \inv g\)) let \(\gamma(g) = q_\Xi(g) \dotminus q_\Xi(1) = (g - 1, \inv g - 1)\), then \(\unit(T, B_1, q_\Xi) = \{g \in T^* \mid g^{-1} = \inv g, \gamma(g) \in \Xi\}\). Note that \(\pi(\gamma(g)) = g - 1\) and \(\rho(\gamma(g)) = \inv g - 1\).

Now we will define odd form rings using axioms. A pair \((R, \Delta)\) will be called an odd form ring, if \(R\) is a ring with involution (non-unital in general), \(\Delta\) is a group with a right \(R^\bullet\)-action, and there are maps \(\phi \colon R \rar \Delta\), \(\pi \colon \Delta \rar R\), \(\rho \colon \Delta \rar R\) such that
\begin{itemize}
\item \(\pi\) is a group homomorphism, \(\pi(u \cdot x) = \pi(u) x\);
\item \(\phi\) is a group homomorphism, \(\phi(\inv x y x) = \phi(y) \cdot x\), \(\phi(x) = 0\) for all \(x \in \herm(R)\);
\item \(u \dotplus \phi(x) = \phi(x) \dotplus u\), \(u \dotplus v = v \dotplus u \dotplus \phi(-\inv{\pi(u)} \pi(v))\);
\item \(\rho(u \dotplus v) = \rho(u) - \inv{\pi(u)} \pi(v) + \rho(v)\), \(\inv{\rho(u)} = \rho(\dotminus u)\), \(\rho(u \cdot x) = \inv x \rho(u) x\);
\item \(\pi(\phi(x)) = 0\), \(\rho(\phi(x)) = x - \inv x\);
\item \(u \cdot (x + y) = u \cdot x \dotplus \phi(\inv y \rho(u) x) \dotplus u \cdot y\).
\end{itemize}
Note that in every odd form ring we have \(\phi(\inv x) = \dotminus \phi(x)\), \(\rho(\dot 0) = 0\), \(\rho(\dotminus u) + \rho(u) + \inv{\pi(u)} \pi(u) = 0\), \(u \cdot 0 = \dot 0\) (the last one holds only for unital \(R\)). An odd form ring \((R, \Delta)\) is called unital if \(R\) is unital and \(u \cdot 1 = u\) for all \(u \in \Delta\), for such odd form rings the identity \(u \cdot (-1) \dotplus u = \varphi(\rho(u))\) holds.

Clearly, any special unital odd form ring is an odd form ring. Conversely, if \((R, \Delta)\) is an odd form ring and \((\pi, \rho) \colon \Delta \rar R \times R\) is injective, then \((R, \Delta)\) is called special (and if \(R\) is unital, then \((R, \Delta)\) is special unital odd form ring as in the definition above). Our odd form rings are be preferable in comparison with the special ones since they behave slightly better under the Stein's relativization and some examples are more natural in the context of odd form rings. In examples \ref{OddOrthOFA} and \ref{StableOFA} we will give natural families of odd form rings without units.

An odd form ideal of an odd form ring \((R, \Delta)\) is a pair \((I, \Gamma)\) such that \(I \leqt R\) is a two-sided ideal, \(\Gamma \leqt \Delta\) is a normal subgroup, \(I = \inv{\,I\,}\), \(\Gamma \cdot R \subseteq \Gamma\), and \(\Gamma_{\mathrm{min}} \leq \Gamma \leq \Gamma_{\mathrm{max}}\), where \(\Gamma_{\mathrm{min}} = \Delta \cdot I \dotplus \phi(I)\) and \(\Gamma_{\mathrm{max}} = \{u \in \Delta \mid \pi(u), \rho(u) \in I\}\) (clearly, \((I, \Gamma_{\mathrm{min}})\) and \((I, \Gamma_{\mathrm{max}})\) are odd form ideals). If \((I, \Gamma) \leqt (R, \Delta)\) is an odd form ideal, then \((R, \Delta) / (I, \Gamma) = (R / I, \Delta / \Gamma)\) is an odd form ring. This does not always holds for special odd form rings, since the factor may not be special. Conversely, if \(f \colon (R, \Delta) \rar (S, \Theta)\) is a morphism of odd form rings (in the obvious sense), then \(\Ker(f) \leqt (R, \Delta)\) is an odd form ideal and \((R, \Delta) / \Ker(f) \cong \mathrm{Im}(f)\).

We say that an odd form ring \((R, \Delta)\) is an odd form algebra over a commutative ring \(K\), if \(R\) is an involution \(K\)-algebra and \((R \rtimes K, \Delta)\) is a unital odd form ring (i.e. the action of \(K^\bullet\) on \(\Delta\) is defined and satisfies appropriate identities). Any odd form algebra \((R, \Delta)\) naturally becomes an ideal in a unital odd form algebra \((R \rtimes K, \Delta)\). Any odd form ring is an odd form \(\mathbb Z\)-algebra.

Let \((R, \Delta)\) be an odd form ring. The unitary group \(\unit(R, \Delta)\) consists of all elements \(g = (\beta(g), \gamma(g)) \in R \times \Delta\) such that \(\alpha(g)^{-1} = \inv{\alpha(g)}\), \(\pi(\gamma(g)) = \beta(g)\), and \(\rho(\gamma(g)) = \inv{\beta(g)}\), where \(\alpha(g) = \beta(g) + 1 \in R \rtimes \mathbb Z\) (we may consider \(\alpha(g)\) as an element of \(R \rtimes K\) if \((R, \Delta)\) is an odd form \(K\)-algebra, or even as an element of \(R\) itself if \(R\) is unital, this will make no difference). Note that the first equation may be written as \(\beta(g) \inv{\beta(g)} + \beta(g) + \inv{\beta(g)} = \inv{\beta(g)} \beta(g) + \beta(g) + \inv{\beta(g)} = 0\). The group operation is given by \(\alpha(gh) = \alpha(g) \alpha(h)\) and \(\gamma(gh) = \gamma(g) \cdot \alpha(h) \dotplus \gamma(h)\). Note that if \((R, \Delta)\) is special unital, then we may identify element \(g \in \unit(R, \Delta)\) with \(\alpha(g) \in R\).

\begin{lemma}\label{AlphaBetaGamma}
There are the following identities.
\begin{align*}
\beta(1) &=
  0;\\
\beta(gh) &=
  \beta(g) \alpha(h) + \beta(h) =
  \beta(g) \beta(h) + \beta(g) + \beta(h);\\
\beta(g^{-1}) &=
  \inv{\beta(g)};\\
\inv{\alpha(g)} x \alpha(g) - x &=
  \inv{\beta(g)} x \beta(g)
  + \inv{\beta(g)} x + x \beta(g);\\
\pi(\gamma(g)) &=
  \beta(g);\\
\rho(\gamma(g)) &=
  \inv{\beta(g)};\\
\gamma(1) &=
  \dot 0;\\
\gamma(gh) &=
  \gamma(g) \cdot \alpha(h) \dotplus \gamma(h) = \gamma(g) \cdot \beta(h)
  \dotplus \gamma(h) \dotplus \gamma(g);\\
\gamma(g^{-1}) &=
  \dotminus \gamma(g) \cdot \inv{\alpha(g)} =
  \dotminus \gamma(g)
  \dotminus \gamma(g) \cdot \inv{\beta(g)}
  \dotminus \phi(\inv{\beta(g)}^2);\\
\beta(\up g h) &=
  \up{\alpha(g)}{\beta(h)} =
  (\beta(g) \beta(h) + \beta(h)) \inv{\alpha(g)};\\
\gamma(\up g h) &=
  (\gamma(g) \cdot \beta(h) \dotplus \gamma(h))
  \cdot \inv{\alpha(g)};\\
\beta([g, h]) &=
  (\up{\alpha(g)}{\beta(h)}
  - \beta(h)) \inv{\alpha(h)} =
  (\beta(g) \beta(h) - \beta(h) \beta(g))
  \inv{\alpha(g)} \inv{\alpha(h)};\\
\gamma([g, h])
  &=
  \bigl(\gamma(g) \cdot \beta(h)
  \dotminus \gamma(h) \cdot \beta(g)
  \dotplus \phi(\inv{\beta(g)} \beta(h))\bigr)
  \cdot \inv{\alpha(g)} \inv{\alpha(h)}.
\end{align*}
\end{lemma}
\begin{proof}
Obvious.
\end{proof}

Now we formulate Stein's relativization in our context. Let \((I, \Gamma) \leqt (R, \Delta)\) be an odd form ideal. The double \((R, \Delta) \times_{(I, \Gamma)} (R, \Delta)\) of \((R, \Delta)\) with respect to \((I, \Gamma)\) is the fiber product of \((R, \Delta)\) with itself over \((R / I, \Delta / \Gamma)\). In other words,
\begin{align*}
R \times_I R &= \{(x, y) \in R \times R \mid x - y \in I\},\\
\Delta \times_\Gamma \Delta &= \{(u, v) \in \Delta \times \Delta \mid u \dotminus v \in \Gamma\}.
\end{align*}

From the point of view of abstract algebra, the double \((R, \Delta) \times_{(I, \Gamma)} (R, \Delta)\) is exactly the congruence on \((R, \Delta)\) induced by \((I, \Gamma)\) and \((R / I, \Delta / \Gamma)\) is the factor by this congruence. The projections from the double to \((R, \Delta)\) will be denoted by \(p_1\) and \(p_2\), the factor-map \((R, \Delta) \rar (R / I, \Delta / \Gamma)\) will be denoted by \(q\), and the diagonal map from \((R, \Delta)\) into the double will be denoted by \(d\) (so \(p_i \circ d = \id\)). The double is canonically isomorphic to \((I \rtimes R, \Gamma \rtimes \Delta)\), where \(I \rtimes R = I \oplus R\) as an abelian group with the multiplication \((x, y) (x', y') = (xx' + xy' + yx', yy')\), \(\Gamma \rtimes \Delta\) is the semi-direct product of groups with the action of \(\Delta\) on \(\Gamma\) via the conjugation (i.e. \((u, v) \dotplus (u', v') = (u \dotplus u' \dotplus \phi(-\inv{\pi(v)} \pi(u')), v \dotplus v')\)), \((u, v) \cdot (x, y) = (u \cdot x \dotplus u \cdot y \dotplus v \cdot x \dotplus \phi(\inv y (\rho(u) + \rho(v)) x), v \cdot y)\), \(\pi(u, v) = (\pi(u), \pi(v))\), \(\phi(x, y) = (\phi(x), \phi(y))\), and \(\rho(u, v) = (\rho(u) - \inv{\pi(u)} \pi(v), \rho(v))\). The isomorphism \((I \rtimes R, \Gamma \rtimes \Delta) \rar (R \times_I R, \Delta \times_\Gamma \Delta)\) is given by \((x, y) \mapsto (x + y, y)\) and \((u, v) \mapsto (u \dotplus v, v)\), and in terms of \((I \rtimes R, \Gamma \rtimes \Delta)\) we have \(p_1(x, y) = x + y\), \(p_1(u, v) = u \dotplus v\), \(p_2(x, y) = y\), \(p_2(u, v) = v\), \(d(x) = (0, x)\), and \(d(u) = (\dot 0, u)\).

Now if \((I, \Gamma) \leqt (R, \Delta)\) is an odd form ideal, then there is a left exact sequence
\[1 \rar \unit(I, \Gamma) \rar \unit(R, \Delta) \rar \unit(R / I, \Delta / \Gamma).\]
Obviously, \(\unit(I, \Gamma) \leqt \unit(R, \Delta)\). It is easy to see that the sequence
\[1 \rar \unit(I, \Gamma) \xrightarrow{p_1^{-1}} \unit(I \rtimes R, \Gamma \rtimes \Delta) \xrightarrow{p_2} \unit(R, \Delta) \rar 1\]
is split exact, i.e. is a semi-direct product (where \(p_1^{-1}\) takes values in \(\Ker(p_2)\)). It follows from the fact that the functor \((R, \Delta) \mapsto \unit(R, \Delta)\) commutes with fibered products.

\section{Idempotents}

Let \((I, \Delta) \leqt (R, \Delta)\) be an odd form ideal in a unital odd form ring and \(e \in R\) be a hermitian idempotent (i.e. \(e = \inv e\) and \(e^2 = e\)). It is easy to see that \((I_e, \Gamma^e_e) \subseteq (I, \Gamma)\) is an odd form subalgebra, where \(I_e = eIe\), \(\Gamma^e_e = \{u \in \Gamma_e \mid \pi(u) \in I_e\}\), and \(\Gamma_e = \Gamma \cdot e\). Usually we will apply this to arbitrary non-unital odd form \(K\)-algebra \((R, \Delta) \leqt (R \rtimes K, \Delta)\). If \((R, \Delta)\) is obtained from a quadratic module \((M, B, q)\), then there is a natural bijection between hermitian idempotents of \(R\) and orthogonal summands \(N \leq M\), under this bijection \((R_e, \Delta^e_e)\) are obtained from \((N, B_{N \times N}, q_N)\).

Let \((I, \Gamma) \leqt (R, \Delta)\) be an odd form ideal of a unital odd form ring, \(e\) and \(e'\) be hermitian idempotents in \(R\). The set \(\unit(e, e'; I, \Gamma)\) consists of pairs \(g = (\beta(g), \gamma(g)) \in I \times \Gamma\) such that \(\beta(g) \in e I e'\), \(\inv{\alpha(g)} \alpha(g) = e'\), \(\alpha(g) \inv{\alpha(g)} = e\), \(\gamma(g) \in \Gamma_{e'}\), \(\pi(\gamma(g)) = \beta(g)\), and \(\rho(\gamma(g)) = \inv{\beta(g)} e'\), where \(\alpha(g) = \beta(g) + e'\). If \(e''\) is the third hermitian idempotent, then there is a multiplication map \(\unit(e, e'; I, \Gamma) \times \unit(e', e''; I, \Gamma) \rar \unit(e, e''; I, \Gamma)\) given by \(\alpha(gg') = \alpha(g) \alpha(g')\) and \(\gamma(gg') = \gamma(g) \cdot \alpha(g') \dotplus \gamma(g')\). The inverses are given by \(\alpha(g^{-1}) = \inv{\alpha(g)}\) and \(\gamma(g^{-1}) = \dotminus \gamma(g) \cdot \inv{\alpha(g)}\). It is easy to see that hermitian idempotents and the sets \(\unit(e, e'; I, \Gamma)\) form a groupoid, \(\unit(e, e; I, \Gamma) \cong \unit(I_e, \Gamma^e_e)\). For example, if \(g \in \unit(I, \Gamma)\) and \(e \in R\) is a hermitian idempotent, then \((\alpha(g) e, \gamma(g) \cdot e) \in \unit(\up {\alpha(g)} e, e; I, \Gamma)\).

Now let us see what happens with quadratic modules \((M, B, q)\) that possess a family of orthogonal hyperbolic summands, \(M = M_0 \perp \mathrm H(P_1) \perp \ldots \perp \mathrm H(P_n)\) for some \(n \geq 0\). Let \(E_i 
\in \End(M)\) be the canonical projections on \(P_i\) for \(0 < |i| \leq n\) and on \(M_0\) for \(i = 0\), then \(E_i\) form a complete system of orthogonal idempotents. Let \(e_i = (E_{-i}^\op, E_i)\) for \(-n \leq i \leq n\), then \(e_i \in T = \{(x^\op, y) \in \End_R(M)^\op \times \End_R(M) \mid B(xm, m') = B(m, ym')\}\). Moreover, these \(e_i\) also form a complete system of orthogonal idempotents, and \(\inv{e_i} = e_{-i}\). Also \((e_i, 0) \in \Xi\) for all \(i \neq 0\), where \(\Xi\) is the odd form parameter considered above, since \(q|_{P_i} = 0\).

We say that \(\eta = (e_-, e_+, q_-, q_+)\) is a hyperbolic pair in an odd form ring \((R, \Delta)\), if \(e_-\) and \(e_+\) are orthogonal idempotents in \(R\), \(\inv{e_+} = e_-\), \(q_-\) and \(q_+\) lie in \(\Delta\), \(\pi(q_-) = e_-\), \(\rho(q_-) = 0\), \(q_- \cdot e_- = q_-\), \(\pi(q_+) = e_+\), \(\rho(q_+) = 0\), \(q_+ \cdot e_+ = q_+\) (it follows that \(q_+ \cdot e_- = \dot 0\) and \(q_- \cdot e_+ = \dot 0\)). Usually we will work with several hyperbolic pairs \(\eta_1, \ldots, \eta_n\), in this case we will use the notation \(\eta_i = (e_{-i}, e_i, q_{-i}, q_i)\) and \(e_{|i|} = e_{-i} + e_i\). Finally, \(e_{a'} = 1 - e_a \in R \rtimes \mathbb Z\) (they may be considered as elements of \(R \rtimes K\) if \((R, \Delta)\) is odd form \(K\)-algebra or as elements of \(R\) if \(R\) is unital), where \(a\) is arbitraty index (integer or of type \(|i|\)).

Hyperbolic pairs \(\eta_1\) and \(\eta_2\) are called orthogonal, if \(e_{|1|}\) and \(e_{|2|}\) are orthogonal idempotents. These pairs are called isomorphic, if there are \(e_{1 2} \in e_1 R e_2\) and \(e_{2 1} \in e_2 R e_1\) such that \(e_1 = e_{1 2} e_{2 1}\) and \(e_2 = e_{2 1} e_{1 2}\). For isomorphic hyperbolic pairs we have \((e_{12} + \inv{e_{21}}, q_{-1} \cdot \inv{e_{21}} \dotplus q_1 \cdot e_{12} \dotminus q_{-2} \dotminus q_2 + \phi(e_2)) \in \unit(e_{|1|}, e_{|2|}; R, \Delta)\).

Finally, \(\eta_1\) and \(\eta_2\) are called Morita equivalent if \(e_{|1|} R e_{|2|} R e_{|1|} = e_{|1|} R e_{|1|}\) and \(e_{|2|} R e_{|1|} R e_{|2|} = e_{|2|} R e_{|2|}\). Morita equivalence means exactly that the unital rings \(e_{|1|} R e_{|1|}\) and \(e_{|2|} R e_{|2|}\) are Morita equivalent with respect to the bimodules \(e_{|2|} R e_{|1|}\) and \(e_{|1|} R e_{|2|}\) under the multiplication maps
\begin{align*}
e_{|1|} R e_{|2|} \otimes_{e_{|2|} R e_{|2|}} e_{|2|} R e_{|1|} &\cong e_{|1|} R e_{|1|},\\
e_{|2|} R e_{|1|} \otimes_{e_{|1|} R e_{|1|}} e_{|1|} R e_{|2|} &\cong e_{|2|} R e_{|2|},
\end{align*}
see \cite{Bass}, chapter II for general Morita theory.

An odd form ring \((R, \Delta)\) has an orthogonal hyperbolic family of rank \(n \geq 0\), if there is a family \(\eta_1, \ldots, \eta_n\) of pairwise orthogonal and Morita equivalent hyperbolic pairs. In this case we set \(e_0 = 1 - (e_{-n} + \ldots + e_{-1} + e_1 + \ldots + e_n) \in R \rtimes \mathbb Z\). If \((R, \Delta)\) is special unital, this is equivalent to the existence of decomposition \(R = \bigoplus_{i = -n}^n e_i R\) into direct summands such that \(e_0 R\) is orthogonal to all \(e_{|i|} R\) for \(0 < i \leq n\), these \(e_{|i|} R\) are pairwise orthogonal and hyperbolic with lagrangians \(e_{-i} R\) and \(e_i R\), and also \(e_{|i|} R\) are isomorphic as right \(R\)-modules to direct summands in \((e_{|j|} R)^N\) for \(N\) big enough and for all \(1 \leq i, j \leq n\).

Let \((R, \Delta)\) be an odd form ring with an orthogonal hyperbolic family. If \(X \leq R\) is a subgroup closed under multiplications on all \(e_i\) from the left and from the right, then we will use the notations \(X_{i j} = e_i X e_j\) and \(X_i = e_i X e_i\). Similarly, if \(\Upsilon \leq \Delta\) is a subgroup closed under right multiplications on all \(e_i\), then we will use the notation \(\Upsilon_i = \Upsilon \cdot e_i\). Let \(e_+ = e_1 + \ldots + e_n\) and \(e_- = \inv{e_+} = e_{-n} + \ldots + e_{-1}\). The expressions \(X_{i'}\), \(X_{|i|}\), \(\Upsilon_+\), and so on have the obvious meaning. Also, \(X_{i *}\) and \(X_{* i}\) mean \(e_i X\) and \(X e_i\). Sometimes we will use the notation \(\alpha_{i j}(g) = e_i \alpha(g) e_j\), \(\beta_{i j}(g) = e_i \beta(g) e_j\), and \(\gamma_i(g) = \gamma(g) \cdot e_i\) for \(g \in \unit(R, \Delta)\).

An orthogonal hyperbolic family \(\eta_1, \ldots, \eta_n\) is called free, if \(\eta_i\) are pairwise isomorphic and these isomorphisms are coherent. In other words, there are \(e_{i j} \in R\) for \(1 \leq i, j \leq n\) such that \(e_{i i} = e_i\), \(e_{i j} e_{j k} = e_{i k}\), \(q_i \cdot e_{i j} = q_j\), and \(q_{-j} \cdot \inv{e_{i j}} = q_{-i}\). We set \(e_{-i, -j} = e_{j i}\) and \(e_{|i|, |j|} = e_{i j} + e_{-i, -j}\) for \(1 \leq i, j \leq n\). In this case there are canonical isomorphisms \(R_{0'} \cong \mat(n, R_{|1|})\), \(R_+ \cong \mat(n, R_1)\), and \(R_- \cong \mat(n, R_{-1})\).

The next proposition gives a construction of the free odd form algebras and shows that every odd form algebra is a factor of the special one.

\begin{prop}\label{FreeOFA}
Let \(K\) be a commutative ring, \(A\) and \(B\) be two abstract sets, \(n \geq 0\) be an integer. Then there is the universal odd form algebra \((R, \Delta)\) over \(K\) with set-theoretical maps \(A \rar R\) and \(B \rar \Delta\) and with a fixed family of \(n\) pairwise orthogonal hyperbolic pairs \(\{(e_{-i}, e_i, q_{-i}, q_i)\}_{i = 1}^n\). More explicitly, if we denote the generators by \(\{x_a\}_{a \in A}\) and \(\{u_b\}_{b \in B}\), then \(R\) is the free \(K\)-module with generators \(r_1 \ldots r_m\) for all \(m > 0\), where each \(r_s\) equals to one of \(x_a\), \(\inv x_a\), \(\pi(u_b)\), \(\inv{\pi(u_b)}\), \(\rho(u_b)\), \(\inv{\rho(u_b)}\), or \(e_i\) for some \(a \in A\), \(b \in B\), \(0 < |i| \leq n\), and there are no consecutive factors of type \(r_s = \inv{\pi(u_b)}\), \(r_{s + 1} = \pi(u_b)\) and of type \(r_s = e_i\), \(r_{s + 1} = e_j\). Similarly, \(\Delta\) has elements \(u_b \cdot r_1 \cdots r_m\) and \(q_i \cdot r_1 \cdots r_m\) for all \(b \in B\), \(0 < |i| \leq n\), \(m \geq 0\), where \(r_1 \cdots r_m\) is the \(K\)-module generator of \(R\) and \(r_1 = e_i\) in the case of \(q_i \cdot r_1 \cdots r_m\), such that they form a \(K\)-module basis of \(\Delta / \phi(R)\). This odd form algebra is special.
\end{prop}
\begin{proof}
Let \(R = \bigoplus_{r_1 \cdots r_m} K r_1 \cdots r_m\), where the direct sum is taken by all products as in the statement. It is a \(K\)-module with an obvious involution. In order to multiply two generators \(r_1 \cdots r_m\) and \(r'_1 \cdots r'_m\), we just concatenate them and use the relations \(e_i e_j = 0\) for \(i \neq j\), \(e_i e_i = e_i\), \(\inv{\pi(u_b)} \pi(u_b) = -\rho(u_b) - \inv{\rho(u_b)}\) in order to eliminate bad consecutive pairs of factors. Clearly, this gives a well-defined multiplication on \(R\), hence \(R\) is an involution \(K\)-algebra.

Now fix an arbitrary linear order on the set \(V = \{u_b \cdot r_1 \cdots r_m, q_i \cdot r_1 \cdots r_m\}\) (where \(r_1 \cdots r_m\) are \(K\)-module generators of \(R\) and \(r_1 = e_i\) for the second type of elements). For any \(v \in V\) we define elements \(\pi(v)\) and \(\rho(v)\) in \(R\) by
\begin{align*}
\pi(u_b \cdot r_1 \cdots r_m) &= \pi(u_b) r_1 \cdots r_m, & \pi(q_i \cdot r_1 \cdots r_m) &= r_1 \cdots r_m,\\
\rho(u_b \cdot r_1 \cdots r_m) &= \inv r_m \cdots \inv r_1 \rho(u_b) r_1 \cdots r_m, & \rho(q_i \cdot r_1 \cdots r_m) &= 0.
\end{align*}
By \(\phi(r)\) we mean the class of \(r\) in \(R / \herm(R)\). Let
\[\Delta = \bigl\{\sum^\cdot_{v \in V} v \cdot k_v \dotplus \phi(r) \mid r \in R, k_v \in K, \text{all but a finite number of }k_v \text{ equal }0\bigr\}.\]
The operations on \(\Delta\) are given by
\begin{align*}
\pi\bigl(\sum^\cdot_{v \in V} v \cdot k_v \dotplus \phi(r)\bigr) &= \sum \pi(v) k_v,\\
\rho\bigl(\sum^\cdot_{v \in V} v \cdot k_v \dotplus \phi(r)\bigr) &= \sum_{v \in V} \rho(v) k_v^2 - \blank\\
&- \sum_{v_1 < v_2} \inv{\pi(v_1)} \pi(v_2) k_{v_1} k_{v_2} + r - \inv r,\\
\bigl(\sum^\cdot_{v \in V} v \cdot k_v \dotplus \phi(r)\bigr) \dotplus \bigl(\sum^\cdot_{v \in V} v \cdot k'_v \dotplus \phi(r')\bigr) &= \sum^\cdot_{v \in V} v \cdot (k_v + k'_v) \dotplus \phi\bigl(r + r' + \blank \\
&+ \sum_{v_1 < v_2} \inv{\pi(v_1)} \pi(v_2) k'_{v_1} k_{v_2} - \sum_v \rho(v) k'_v k_v\bigr).
\end{align*}
The semigroup \(R^\bullet\) acts on \(\Delta\) from the right in the obvious way. It is easy to see that \(\Delta\) is a group and all axioms on the odd form parameter are satisfied. By construction, \((R, \Delta)\) has the universal property. Clearly, if \(\pi(u) = 0\) for some \(u \in \Delta\), then \(u \in \phi(R)\), hence \((R, \Delta)\) is special.
\end{proof}

Existence of an orthogonal hyperbolic family helps to somewhat simplify the odd form parameter \(\Delta\). Let \(\Delta^i = q_i \cdot R \cong R_{i *}\) for all \(i \neq 0\) and \(\Delta^0 = \{u \in \Delta \mid \pi(u) \in R_{0 *}\}\). The sets \(\Delta^i\) are actually subgroups, \(\phi(R) \leq \Delta^0\), \([\Delta^i, \Delta^j] = \dot 0\) for \(i \neq -j\) (where \([u, v] = u \dotplus v \dotminus u \dotminus v\)), and \([\Delta^i, \Delta^{-i}] \leq \Delta^0\). Moreover, any element \(u \in \Delta\) may be decomposed uniquely as \(u = u^{-n} \dotplus \ldots \dotplus u^n\) with \(u^i \in \Delta^i\). It is easy to see that \(u^i = q_i \cdot \pi(u)\) for all \(i \neq 0\), \((u \cdot x)^i = u^i \cdot x\) for all \(i\), \((u \dotplus v)^i = u^i \dotplus v^i\) for all \(i \neq 0\), and \((u \dotplus v)^0 = u^0 \dotplus v^0 \dotplus \phi(\inv{\pi(v)} e_+ \pi(u))\). For other operations we have \(\pi(u^i) = e_i \pi(u)\) for all \(i\), \(\rho(u^i) = 0\) for \(i \neq 0\), \(\rho(u^0) = \rho(u) + \inv{\pi(u)} e_+ \pi(u)\), \(\phi(x)^i = \dot 0\) for \(i \neq 0\), and \(\phi(x)^0 = \phi(x)\). In the special unital case \(\Delta^0\) is an odd form parameter in \(\Heis(R_{0 *}, B_1|_{R_{0 *} \times R_{0 *}})\). We may always work with \(\Delta^0\) instead of \(\Delta\), as we will now see. If \((I, \Gamma) \leq (R, \Delta)\) is an odd form ideal, then \(\Gamma = \sum^\cdot_i (\Gamma \cap \Delta^i)\) and \(\Gamma \cap \Delta^i = q_i \cdot I\) for \(i \neq 0\), hence \(\Gamma\) is determined by \(\Gamma^0 = \Gamma \cap \Delta^0\). Clearly, \(\Gamma^0 \cdot R \subseteq \Gamma^0\) and \(\Gamma^0_{\mathrm{min}} \leq \Gamma^0 \leq \Gamma^{\mathrm{max}}\) are necessary and sufficient conditions on \(\Gamma^0 \leq \Delta^0\) to generate an odd form ideal \((I, \Gamma)\) if \(I \leqt R\) is fixed, where \(\Gamma^0_{\mathrm{min}} = \Delta^0 \cdot I \dotplus \phi(I)\) and \(\Gamma^0_{\mathrm{max}} = \{u \in \Delta^0 \mid \pi(u), \rho(u) \in I\}\).

If \(g = (\alpha(g), \gamma(g)) \in \unit(R, \Delta)\), then \(\gamma^i(g) = q_i \cdot \beta(g)\) for \(i \neq 0\), hence \(g\) is determined by \(\alpha(g)\) and \(\gamma^0(g)\). Let \(\gamma^{\circ}(g) = \gamma^0(g) \dotplus \phi(e_+ \beta(g))\), then (where, as usual, \(\alpha(g) = \beta(g) + 1\))
\begin{equation*}
\begin{split}
\unit(R, \Delta) \cong \{g = (\beta(g), \gamma^{\circ}(g)) \in R \times \Delta^0 \mid \phantom{} &\alpha(g)^{-1} = \inv{\alpha(g)}, \pi(\gamma^{\circ}(g)) = e_0 \beta(g), \\
&\rho(\gamma^{\circ}(g)) = \inv{\alpha(g)} e_+ \alpha(g) - e_+ + \inv{\beta(g)} e_0\}.
\end{split}
\end{equation*}
In the special unital case \(\gamma^\circ(g) = (e_0, e_+) \cdot \alpha(g) \dotminus (e_0, e_+)\). The group operation is given by
\begin{align*}
\alpha(gh) &= \alpha(g) \alpha(h),\\
\gamma^0(gh) &= \gamma^0(g) \cdot \alpha(h) \dotplus \gamma^0(h) + \phi(\inv{\beta(h)} e_+ \beta(g) \alpha(h)),\\
\gamma^{\circ}(gh) &= \gamma^{\circ}(g) \cdot \alpha(h) \dotplus \gamma^{\circ}(h).
\end{align*}

\section{Elementary transvections}

In this section \((R, \Delta)\) is an odd form ring with an orthogonal hyperbolic family of rank \(n\). Then there are elements in the unitary group \(\unit(R, \Delta)\) of a particularly simple structure. An elementary transvection of a short root type is an element \(T_{i j}(x) \in \unit(R, \Delta)\) such that
\begin{align*}
\beta(T_{i j}(x)) &= x - \inv x, &
\gamma^0(T_{i j}(x)) &= \dotminus \phi(x), &
\gamma^{\circ}(T_{i j}(x)) &= \begin{cases}
\phi(\inv x), &\text{if } i < 0 < j;\\
\phi(x), &\text{if } j < 0 < i;\\
\dot 0, &\text{if } 0 < ij
\end{cases}
\end{align*}
for any \(i \neq 0\), \(j \neq 0\), \(i \neq \pm j\), and \(x \in R_{i j}\). An elementary transvection of an ultrashort root type is an element \(T_i(u) \in \unit(R, \Delta)\) such that
\begin{align*}
\beta(T_i(u)) &= \rho(u) + \pi(u) - \inv{\pi(u)}, \\
\gamma^0(T_i(u)) &= u \dotminus \phi(\rho(u) + \pi(u)), \\
\gamma^{\circ}(T_i(u)) &= \begin{cases}
u, &\text{if } i < 0;\\
u \dotminus \phi(\rho(u) + \pi(u)), &\text{if } 0 < i
\end{cases}
\end{align*}
for any \(i \neq 0\) and \(u \in \Delta^0_i = \Delta^0 \cdot e_i\) (an elementary transvection of a long root type is \(T_i(\phi(x))\) for some \(x \in R_{-i, i}\)). It can be easily seen that all these elements are indeed in the unitary group. The elementary unitary group is
\[\eunit(R, \Delta) = \langle T_{i j}(x), T_k(u) \mid i \neq \pm j; i, j, k \neq 0; x \in R_{i j}; u \in \Delta^0_k \rangle.\]

An elementary dilation is an element \(D_i(a) \in \unit(R, \Delta)\) such that
\begin{align*}
\beta(D_i(a)) &= a + \inv a^{-1} - e_{|i|}, & \gamma^0(D_i(a)) &= \phi(e_i - a), &
\gamma^\circ(D_i(a)) &= \dot 0
\end{align*}
for any \(i \neq 0\) and \(a \in R_i^*\). Also, \(D_0(g) = g\) for \(g \in \unit(R_0, \Delta_0^0)\).

\begin{lemma}\label{TransvectionRelations}
Elementary transvections and dilations satisfy the following relations:
\begin{itemize}
\item \(T_{i j} \colon R_{i j} \rar \unit(R, \Delta)\), \(T_i \colon \Delta^0_i \rar \unit(R, \Delta)\), \(D_i \colon R_i^* \rar \unit(R, \Delta)\) for \(i \neq 0\), and \(D_0 \colon \unit(R_0, \Delta_0^0) \rar \unit(R, \Delta)\) are group homomorphisms;
\item \(T_{i j}(x) = T_{-j, -i}(-\inv x)\), \(D_i(a) = D_{-i}(\inv a^{-1})\) for \(i \neq 0\);
\item \([D_i(a), D_j(b)] = 1\) for \(i \neq \pm j\);
\item \([D_i(a), T_{j k}(x)] = 1\) for \(j \neq \pm i \neq k\)
\item \([D_i(a), T_j(u)] = 1\) for \(0 \neq i \neq \pm j\);
\item \(\up{D_i(a)}{T_{ij}(x)} = T_{ij}(ax)\);
\item \(\up{D_i(a)}{T_{-i}(u)} = T_{-i}(u \cdot \inv a)\);
\item \(\up{D_0(g)}{T_i(u)} = T_i(\gamma(g) \cdot \pi(u) \dotplus u) = T_i((\gamma(g) \cdot \pi(u) \dotplus u) \cdot \inv{\alpha(g)})\);
\item \([T_{i j}(x), T_{k l}(y)] = 1\) for \(i \neq l \neq -j \neq -k \neq i\);
\item \([T_{i j}(x), T_{j k}(y)] = T_{i k}(xy)\) for \(i \neq \pm k\);
\item \([T_{-i, j}(x), T_{j i}(y)] = T_i(\phi(xy))\);
\item \([T_i(u), T_j(v)] = T_{-i, j}(-\inv{\pi(u)} \pi(v))\) for \(i \neq \pm j\);
\item \([T_i(u), T_{j k}(x)] = 1\) for \(j \neq i \neq -k\);
\item \([T_i(u), T_{i j}(x)] = T_{-i, j}(\rho(u) x)\, T_j(\dotminus u \cdot (-x))\).
\end{itemize}
\end{lemma}
\begin{proof}
The formulas with dilations are trivial. We prove other relations without the assumption that the hyperbolic pairs are Morita equivalent. Without loss of generality we may assume that the odd form ring is free (say, over \(\mathbb Z\)), hence special by proposition \ref{FreeOFA}. But since our odd form ring is special, it suffices to check that values of \(\beta\) on both sides of each identity coincide. By lemma \ref{AlphaBetaGamma}, we have \(\beta([g, h]) = (\beta(g) \beta(h) - \beta(h) \beta(g)) \inv{\alpha(g)} \inv{\alpha(h)}\), hence the proof reduces to direct routine calculations.
\end{proof}

Let \((I, \Gamma) \leq (R, \Delta)\) be an odd form ideal. Clearly, \((R / I, \Delta / \Gamma)\) also has an orthogonal hyperbolic family such that the factor-morphism preserves the family. The relative elementary unitary group is
\[\eunit(R, \Delta; I, \Gamma) = \up{\eunit(R, \Delta)}{\langle T_{i j}(x), T_k(u) \mid x \in I_{i j}, u \in \Gamma^0_k \rangle}.\]
In other words, it is the smallest subgroup of \(\unit(R, \Delta)\) normalized by the elementary unitary group that contains all elementary transvections from \(\unit(I, \Gamma)\) (note that \((I, \Gamma)\) does not necessarily contains the orthogonal hyperblic family). Clearly, this is a subgroup of \(\unit(I, \Gamma)\).

\begin{lemma}\label{Perfect}
If \((I, \Gamma) \leqt (R, \Delta)\) is an odd form ideal and \((R, \Delta)\) has a fixed orthogonal hyperbolic family of rank \(n \geq 3\), then
\[[\eunit(R, \Delta), \eunit(R, \Delta; I, \Gamma)] = \eunit(R, \Delta; I, \Gamma).\]
In particular, \(\eunit(R, \Delta)\) is perfect.
\end{lemma}
\begin{proof}
Clearly, the left hand side is contained in the right one. It suffices to prove that every elementary transvection from \(\unit(I, \Gamma)\) lies in the commutant. But this follows from lemma \ref{TransvectionRelations} and from \(R_{|i|, |j|} R_{|j|, |i|} = R_{|i|}\) for \(i, j \neq 0\): we have \(T_{i j}(x) = \prod_{r_+} [T_{i k}(x p_{r_+}), T_{k j}(q_{r_+})]\, \prod_{r_-} [T_{i, -k}(x p_{r_-}), T_{-k, j}(q_{r_-})]\) if \(\sum_{r_+} p_{r_+} q_{r_+} + \sum_{r_-} p_{r_-} q_{r_-} = e_j\), \(k \neq 0\), and \(k\) is different from \(\pm i, \pm j\). Similarly for the long transvections, i.e. for \(T_i(\phi(I_{-i, i}))\). For the ultrashort transvections we have \(T_i(u) \in \langle T_{j i}(I_{j i}), T_{-j, i}(I_{-j, i}), T_i(\phi(I_{-i, i})), [T_j(\Gamma^0_j), T_{j i}(R_{j i})], [T_{-j}(\Gamma^0_{-j}), T_{-j, i}(R_{-j, i})] \rangle\) if \(j \neq 0\) and \(j \neq \pm i\), since \(\Gamma^0_i = \Gamma^0_j \cdot R_{j i} \dotplus \Gamma^0_{-j} \cdot R_{-j, i} \dotplus \phi(R_{-i, i})\).
\end{proof}

\begin{lemma}\label{Relativization}
Let \((R, \Delta)\) be an odd form ring with an orthogonal hyperbolic family, \((I, \Gamma) \leqt (R, \Delta)\). Then the sequence
\[1 \rar \eunit(R, \Delta; I, \Gamma) \xrightarrow{p_1^{-1}} \eunit(I \rtimes R, \Gamma \rtimes \Delta) \xrightarrow{p_2} \eunit(R, \Delta) \rar 1\]
is split exact, i.e. a semi-direct product, where \(p_1^{-1}\) takes values in \(\Ker(p_2)\). The section is given by \(d\).
\end{lemma}
\begin{proof}
Let \(N = p_1^{-1}(\eunit(R, \Delta; I, \Gamma))\) and \(G = d(\eunit(R, \Delta))\), then \(N \cap G = 1\), \([G, N] \leq N\) (since \(\eunit(R, \Delta; I, \Gamma)\) is normalized by \(\eunit(R, \Delta)\)), and every generator of \(\eunit(I \rtimes R, \Gamma \rtimes \Delta)\) lies in \(NG\). Hence \(\eunit(I \rtimes R, \Gamma \rtimes \Delta) = N \rtimes G\).
\end{proof}

The next lemma shows what happens when we change the number of hyperbolic pairs.

\begin{lemma}\label{HyperbolicGluing}
Let \((R, \Delta)\) be an odd form ring with an orthogonal hyperbolic family \(\eta_1, \ldots, \eta_n\), and \(T_{ij}(*)\), \(T_i(*)\), \(D_i(*)\) be the groups of corresponding elementary translations and dilations. If \(T'_{ij}(*)\), \(T'_i(*)\), \(D'_i(*)\) are the groups of elemenary translations and dilations obtained from the orthogonal hyperbolic family \(\eta_2, \ldots, \eta_n\), then
\begin{itemize}
\item \(T_{ij}(*) = T'_{ij}(*)\) and \(D_i(*) = D'_i(*)\) for \(2 \leq |i|, |j| \leq n\);
\item the multiplication map induces a bijection \(T_{1 i}(*) \times T_{-1, i}(*) \times T_i(*) \cong T'_i(*)\) for \(2 \leq |i| \leq n\);
\item \(\langle T_1(*), T_{-1}(*), D_1(*), D_0(*) \rangle \leq D'_0(*)\).
\end{itemize}
Similarly, if \(T''_{ij}(*)\), \(T''_i(*)\), \(D''_i(*)\) are the groups obtained from the orthogonal hyperbolic family \(\eta_1, \ldots, \eta_{n - 2}, \eta_{n - 1} + \eta_n\), where \(\eta_{n - 1} + \eta_n = (e_{1 - n} + e_{-n}, e_{n - 1} + e_n, q_{1 - n} \dotplus q_{-n}, q_{n - 1} \dotplus q_n)\), then
\begin{itemize}
\item \(T_{ij}(*) = T''_{ij}(*)\), \(T_i(*) = T''_i(*)\), and \(D_i(*) = D''_i(*)\) for \(|i|, |j| < n - 1\);
\item the multiplication map induces a bijection \(T_{i, n - 1}(*) \times T_{i n}(*) \cong T''_{i, n - 1}(*)\) for \(|i| < n - 1\);
\item \(\langle T_{n, n - 1}(*), T_{n - 1, n}(*), D_n(*), D_{n - 1}(*) \rangle \rangle \leq D''_{n - 1}(*)\).
\end{itemize}
\end{lemma}
\begin{proof}
Clear.
\end{proof}

At the end of this section we translate notation from \cite{OddStrucVor} into our new general setting. In that paper every augmented level \(L\) is actually an odd form ideal of an augmented level \(\langle L, \lfloor L_0 \rfloor \rangle\), which itself is an odd form algebra (there we used symbols \(\tr\) and \(\varphi\) instead of \(\rho\) and \(\phi\)). The following proposition shows the converse: every odd form algebra \((R, \Delta)\) with an odd form ideal \((I, \Gamma)\) may be reduced preserving the elementary unitary group in such a way that the result satisfies all crucial properties of augmented levels.

\begin{prop}\label{OddFormRingReduction}
Let \((R, \Delta)\) be an odd form \(K\)-algebra. Then there is an odd form subalgebra \((R', \Delta') \subseteq (R, \Delta)\) containing the orthogonal hyperbolic family such that \(\unit(R', \Delta') = \unit(R, \Delta)\), \(R'_{0', 0'} = R_{0', 0'}\), \(R'_{0, 0'} = \pi({\Delta'}^0_{0'})\), and \({\Delta'}^0_{0'} = \Delta^0_{0'}\). If \((I, \Gamma) \leqt (R, \Delta)\) is an odd form ideal, then there are \((R', \Delta') \subseteq (R, \Delta)\) and odd form ideal \((I', \Gamma') \leqt (R', \Delta')\) such that \((I', \Gamma') \subseteq (I, \Gamma)\), \(\unit(I', \Gamma') = \unit(I, \Gamma)\), \(I'_{0', 0'} = I_{0', 0'}\), \(I'_{0, 0'} = \pi({\Gamma'}^0_{0'})\), and \({\Gamma'}^0_{0'} = \Gamma^0_{0'}\). Moreover, there is a functorial construction of \((R', \Delta'; I', \Gamma')\).

If \((R, \Delta)\) has an orthogonal hyperbolic family of rank \(n \geq 3\), then there is also a non-functorial odd form subalgebra \((R'', \Delta'')\) such that \((R'', \Delta'')\) contains the orthogonal hyperbolic family, \(\eunit(R'', \Delta'') = \eunit(R, \Delta)\), \(R''_{0', 0'} = R_{0', 0'}\), \(R''_{0, 0'} = \pi({\Delta''}^0_{0'})\), \(R''_{0, 0} = R''_{0, 0'} R''_{0', 0}\), \({\Delta''}^0_{0'} = \Delta^0_{0'}\), and \({\Delta''}^0_0 = {\Delta''}^0_{0'} \cdot R''_{0', 0} \dotplus \phi(R''_{0, 0})\). Similarly, if \((I, \Gamma) \leqt (R, \Delta)\), then there is \((I'', \Gamma'') \leqt (R'', \Delta'')\) such that \(\eunit(R'', \Delta''; I'', \Gamma'') = \eunit(R, \Delta; I, \Gamma)\), \(I''_{0', 0'} = I_{0', 0'}\), \(I''_{0, 0'} = \pi({\Gamma''}^0_{0'})\), \(I''_{0, 0} = I''_{0, 0'} R''_{0', 0} + R''_{0, 0'} I''_{0', 0}\), \({\Gamma''}^0_{0'} = \Gamma^0_{0'}\), and \({\Gamma''}^0_0 = {\Gamma''}^0_{0'} \cdot R''_{0', 0} \dotplus {\Delta''}^0_{0'} \cdot I''_{0', 0} \dotplus \phi(I''_{0, 0})\).
\end{prop}
\begin{proof}
Let \(R' \subseteq R\) be the subalgebra generated by \(\beta(g)\) for all \(g \in \unit(R, \Delta)\) and \(\Delta' \leq \Delta\) be the subgroup generated by \(\phi(R')\) and \(\gamma(g) \cdot K\) for all \(g \in \unit(R, \Delta)\). The definition of the unitary group shows that \((R', \Delta')\) is an odd form \(K\)-subalgebra and \(\unit(R', \Delta') = \unit(R, \Delta)\). Moreover, since elementary transvections are in \(\unit(R, \Delta)\), the orthogonal hyperbolic family lies in \((R', \Delta')\) and the required identities hold.

Similarly, if \((I, \Gamma) \leqt (R, \Delta)\), then we may set \(I' = \langle \beta(g) \mid g \in \unit(I, \Gamma) \rangle_{R'} \leqt R'\) and
\[\Gamma' = \langle \gamma(g), \gamma(g) \cdot y, \phi(x), u \cdot x \mid g \in \unit(I, \Gamma), x \in I', u \in \Delta', y \in R' \rangle \leqt \Delta'.\]
Then \((I', \Gamma') \leqt (R', \Delta')\), \(\unit(I', \Gamma') = \unit(I, \Gamma)\), and the identities hold.

For the second part, the objects \(R''\), \(\Delta''\), \(I''\), and \(\Gamma''\) are uniquely determined by the equalities. It is easy to see that they satisfy all conditions.
\end{proof}

Odd form algebras and ideals as in the second part of the previous proposition will be called reduced. This notion clearly depends on the choice of orthogonal hyperbolic family.

\section{Finiteness conditions}

First of all, we say that \((R, \Delta)\) is a semi-local odd form ring if \(R\) is a unital semi-local ring, i.e. the factor of \(R\) by its Jacobson radical is a finite product of matrix algebras over division rings. Note that if \((R, \Delta)\) is semi-local and \((I, \Gamma) \leqt (R, \Delta)\), then \((I \rtimes R, \Gamma \rtimes \Delta)\) is also semi-local.

Next, \((R, \Delta)\) is semi-simple, if \(R\) is a unital semi-simple ring (and, in particular, unital semi-local) and \((R, \Delta)\) is special. If \((R, \Delta)\) is semi-local, then there is an odd form ideal \((I, \Gamma) \leqt (R, \Delta)\) such that \(I\) is the Jacobson radical of \(R\) and \(\Gamma = \Gamma_{\mathrm{max}}\), hence \((R / I, \Delta / \Gamma)\) is semi-simple. In general if \((R, \Delta)\) is arbitrary unital odd form ring and \(R = R' \times R''\), then there is unique decomposition \((R, \Delta) = (R', \Delta') \times (R'', \Delta'')\), where \(\Delta' = \Delta \cdot R' \dotplus \phi(R')\) and \(\Delta'' = \Delta \cdot R'' \dotplus \phi(R'')\). In particular, every semi-simple odd form ring is a product of simple artinian odd form rings. A special odd form ring \((R, \Delta)\) is simple artinian, if either \(R \cong \mat(n, D)^\op \times \mat(n, D)\) for some \(n > 0\) and division ring \(D\) (with an obvious involution) or \(R \cong \mat(n, D)\) with some involution for a division ring \(D\) and \(n > 0\).

If \(K \rar K'\) is a homomorphism of commutative rings and \((R, \Delta)\) is an odd form \(K\)-algebra, then \((R, \Delta) \otimes_K K' = (R \otimes_K K', \Delta \otimes_K K')\) is an odd form \(K'\)-algebra, where \(\Delta \otimes_K K'\) is the abstract group generated by elements \(u \otimes a\) and \(\phi(x)\) for \(u \in \Delta\), \(x \in R \otimes_K K'\), \(a \in K'\) with the relations
\begin{itemize}
\item \(\phi(x + y) = \phi(x) \dotplus \phi(y)\), \(\phi(x) = 0\) for \(x \in \herm(R \otimes_K K')\);
\item \(u \otimes a \dotplus \phi(x) = \phi(x) \dotplus u \otimes a\), \(u \otimes a \dotplus v \otimes b = v \otimes b \dotplus u \otimes a \dotplus \phi(-\inv{\pi(u)} \pi(v) \otimes a b)\);
\item \((u \dotplus v) \otimes a = u \otimes a \dotplus v \otimes a\), \((u \cdot k) \otimes a = u \otimes ka\) for \(k \in K\);
\item \(u \otimes (a + b) = u \otimes a \dotplus u \otimes b \dotplus \phi(\rho(u) \otimes ab)\);
\item \(\phi(x) \otimes a = \phi(x \otimes a^2)\) for \(x \in R\).
\end{itemize}
It can be directly checked that \((R, \Delta) \otimes_K K'\) is indeed an odd form algebra over \(K'\) if the operations are defined using the axioms for odd form rings and
\begin{align*}
(u \otimes a) \cdot (x \otimes b) &= (u \cdot x) \otimes ab, &
\rho(u \otimes a) &= \rho(u) \otimes a^2, &
\pi(u \otimes a) &= \pi(u) \otimes a.
\end{align*}

Clearly, \((R, \Delta) \rar (R, \Delta) \otimes_K K', x \mapsto x \otimes 1, u \mapsto u \otimes 1\) is a morphism of odd form rings (and of odd form \(K\)-algebras). Also, if \((R, \Delta)\) is the free odd form algebra over \(K\) as in proposition \ref{FreeOFA}, then \((R, \Delta) \otimes_K K'\) is the free odd form algebra over \(K'\) with the same generators. If \(K' = S^{-1} K\) for some multiplicative subset \(S \leq K^\bullet\), then \(S^{-1} (R, \Delta) = (R, \Delta) \otimes_K S^{-1} K\) is called the localization of \((R, \Delta)\) by \(S\).

There is an important class of quasi-finite odd form algebras. Recall that \(K\)-algebra \(R\) is called quasi-finite, if \(R\) is direct limit of finite \(K\)-algebras (i.e. that are finitely generated \(K\)-modules). This is equivalent to existence of  finitely generated commutative rings \((K_i)_{i \in I}\) and finite \(K_i\)-algebras \(R_i\) for some directed set \(I\) such that \((K, R) = \varinjlim_i (K_i, R_i)\) (i.e. \(K\) is the direct limit of rings \(K_i\), \(R\) is the direct limit of rings \(R_i\), and \(R_i \rar R_j\) are \(K_i\)-linear for all \(i \leq j\)). We say that odd form \(K\)-algebra \((R, \Delta)\) is quasi-finite if \(R\) is quasi-finite \(K\)-algebra (and finite, if \(R\) is a finite \(K\)-algebra and \(\Delta / \phi(R)\) is a finite \(K\)-module). Note that if \((R, \Delta)\) is quasi-finite over \(K\) and \((I, \Gamma) \leqt (R, \Delta)\), then \((I \rtimes R, \Gamma \rtimes \Delta)\) is also quasi-finite. Properties of being quasi-finite and finite are preserved under the extension of scalars. Also, a finite algebra over a semi-local commutative ring is semi-local.

\begin{lemma}\label{QuasiFinite}
Let \((R, \Delta)\) be an odd form \(K\)-algebra. Then the following are equivalent:
\begin{enumerate}
\item \((R, \Delta)\) is quasi-finite over \(K\);
\item \((R, \Delta)\) is a direct limit of odd form subalgebras \((R_i, \Delta_i)\) such that \(R_i\) are finite \(K\)-modules and \(\Delta_i / \phi(R_i)\) are finitely generated \(K\)-modules;
\item There are finitely generated commutative rings \((K_i)_{i \in I}\) and odd form \(K_i\)-algebras \((R_i, \Delta_i)\) such that \(I\) is a directed set, \((K, R, \Delta) = \varinjlim_i (K_i, R_i, \Delta_i)\), and \((R_i, \Delta_i)\) are finite odd form \(K_i\)-algebras;
\item \((R, \Delta)\) is quasi-finite over a subring \(K_0 \subseteq K\) such that the extension \(K / K_0\) is integral;
\item \((R \rtimes K, \Delta)\) is quasi-finite over \(K\).
\end{enumerate}
\end{lemma}
\begin{proof}
The implications \((3) \rrar (2) \rrar (1)\), \((4) \Leftrightarrow (1)\), and \((1) \Leftrightarrow (5)\) are obvious. Suppose that \((R, \Delta)\) is quasi-finite over \(K\), then \((K, R) = \varinjlim_{i \in I} (K_i, R_i)\) for some finitely generated commutative rings \(K_i\) and finite \(K_i\)-algebras \(R_i\). We may assume that \(K_i \subseteq K\) and \(R_i \subseteq R\) are closed under the involution. Let \(\Delta_i = \{u \in \Delta \mid \pi(u), \rho(u) \in R_i\}\), then \((R_i, \Delta_i)\) are odd form rings over \(K_i\) and \(\Delta = \varinjlim_i \Delta_i\). Clearly, right \(R_i\)-modules \(\Delta_i / \phi(R_i)\) are direct limits of finitely generated submodules \(\Delta_{i, j} / \phi(R_i)\) for some \(j \in J_i\) (they are also finitely generated \(K_i\)-modules), hence \((R_i, \Delta_{i, j})\) are finite odd form \(K_i\)-algebras and \((3)\) follows. Note that if \((R, \Delta)\) is special, then we may choose \((R_i, \Delta_{i, j})\) to be special too.
\end{proof}

Let \((R, \Delta)\) be an odd form algebra over \(K\). We say that an orthogonal hyperbolic family \(\eta_1, \ldots, \eta_n\) is Morita complete, if \(n \geq 1\) and for every orthogonal hyperbolic family \(\eta'_1, \ldots, \eta'_n\) in \((R, \Delta)\) we have \(e'_{|1|} R e_{|1|} R e'_{|1|} = e'_{|1|} R e'_{|1|}\). If \((R, \Delta)\) is an odd form algebra with a Morita complete orthogonal hyperbolic family and \(\sigma \in \Aut(R, \Delta)\), then \(\up\sigma{\eta_1}\) is Morita equivalent to \(\eta_1\). The next lemma gives several criteria for Morita completeness that work for semi-local odd form rings and for classical odd form algebras over commutative rings including their twisted forms, see examples below.

\begin{lemma}\label{MoritaCriteria}
Let \((R, \Delta)\) be an odd form algebra over a commutative ring \(K\) with an orthogonal hyperbolic family of rank \(n \geq 1\). Then:
\begin{enumerate}
\item If the family is Morita complete in \((R, \Delta)_{\mathfrak m}\) for every maximal \(\mathfrak m \leqt K\), then it is Morita complete in \((R, \Delta)\).
\item If the family is Morita complete in \((R, \Delta) \otimes_K K_j\) for an fppf-covering \((K_j)_{j \in J}\) of \(K\), then it is Morita complete in \((R, \Delta)\).
\item If \((I, \Gamma) \leqt (R, \Delta)\) is an odd form ideal, \(I\) is contained in the Jacobson radical of \((R, \Delta)\), and the family is Morita complete in \((R / I, \Delta / \Gamma)\), then it is Morita complete in \((R, \Delta)\).
\item If \((R, \Delta) = \prod_{j = 1}^N (R_j, \Delta_j)\), then the family is Morita complete in every \((R_j, \Delta_j)\) if and only if it is Morita complete in \((R, \Delta)\).
\item If \((R, \Delta)\) is simple artinian and the family is non-zero (equivalently, \(e_1 \neq 0\)), then the family is Morita complete.
\item If \((R, \Delta)\) is simple artinian with zero orthogonal hyperbolic family, \(R \cong \mat(k, D)^\op \times \mat(k, D)\), and \(k < n\), then the family is Morita complete.
\item If \((R, \Delta)\) is simple artinian with zero orthogonal hyperbolic family, \(R \cong \mat(k, D)\), and \(k < 2n\), then the family is Morita complete.
\end{enumerate}
\end{lemma}
\begin{proof}
Almost all of these is obvious, since the condition for Morita completeness may be written as \(e'_{|1|} \in e'_{|1|} R e_{|1|} R e'_{|1|}\). In \((6)\) and \((7)\) note that if \(\eta'_1, \ldots, \eta'_n\) is non-zero orthogonal hyperbolic family, then \(e'_{0'} R e'_{0'}\) is a subring of type \(\mat(l, D)^\op \times \mat(l, D)\) or \(\mat(l, D)\) (up to an isomorphism) with a lot of orthogonal idempotents contradicting the bound on \(k\).
\end{proof}

\section{Projective unitary groups}

Let \((R, \Delta)\) and \((S, \Theta)\) be odd form \(K\)-algebras. The set of all homomorphisms from \((R, \Delta)\) to \((S, \Theta)\) will be denoted as \(\Hom(R, \Delta; S, \Theta)\). The action of \(f \in \Hom(R, \Delta; S, \Theta)\) on elements \(x \in R\) and \(u \in \Delta\) will often be written as \(\up f x\) and \(\up f u\).

The projective unitary group \(\punit(R, \Delta)\) of an odd form algebra \((R, \Delta)\) is the group of odd form algebra automorphisms, i.e. \(\punit(R, \Delta) = \Aut(R, \Delta)\). As we will show below, the projective unitary group coincides with the corresponding projective group scheme for all classical Chevalley groups. Note also that if \((R, \Delta)\) is reduced (with respect to some orthogonal hyperbolic family of rank \(n \geq 3\)), then \(\punit(R, \Delta) = \punit(R \rtimes K, \Delta)\). If \((I, \Gamma) \leqt (R, \Delta)\), then the relative projective unitary group is
\[\punit(R, \Delta; I, \Gamma) = \{\sigma \in \punit(R, \Delta) \mid \up \sigma x - x \in I \text{ and } \up \sigma u \dotminus u \in \Gamma \text{ for all } x \in R, u \in \Delta\}.\]
In other words, the relative projective unitary group is the centralizer of \((R / I, \Delta / \Gamma)\) in the stabilizer subgroup \(\punit(R, \Delta)_{(I, \Gamma)} = \{\sigma \in \punit(R, \Delta) \mid \up \sigma I = I, \up \sigma \Gamma = \Gamma\}\).

Finally, let
\begin{align*}
\punit^{\mathrm{split}}(R \times_I R, \Delta \times_\Gamma \Delta) = \{&\sigma \in \punit(R \times_I R, \Delta \times_\Gamma \Delta) \mid\\
& \sigma = \sigma_1 \times \sigma_2 \text{ for some } \sigma_i \in \punit(R, \Delta)\}.
\end{align*}
In the definition \(\sigma_i\) are uniquelly determined, \(\sigma_i = p_i(\sigma) = p_i \circ \sigma \circ d\). Also there is the splitting map \(d \colon \punit(R, \Delta) \rar \punit^{\mathrm{split}}(I \rtimes R, \Gamma \rtimes \Delta), \sigma \mapsto \sigma \times \sigma\).

\begin{lemma}\label{ProjectiveUnitaryGroups}
Let \((R, \Delta)\) and \((S, \Theta)\) be odd form \(K\)-algebras, \((I, \Gamma) \leqt (R, \Delta)\). Then:
\begin{enumerate}
\item If \(f \in \Hom(S, \Theta; R, \Delta)\), then \(\up f {\unit(S, \Theta)} \leq \unit(R, \Delta)\) and \(\up f {\unit(S, \Theta; f^{-1}(I), f^{-1}(\Gamma))} \leq \unit(I, \Gamma)\).
\item If \(\sigma \in \punit(R, \Delta)\), then \(\up \sigma {\unit(R, \Delta)} = \unit(R, \Delta)\), \(\up \sigma {\unit(I, \Gamma)} = \unit(\up \sigma I, \up \sigma \Gamma)\), and \(\up \sigma {\punit(R, \Delta; I, \Gamma)} = \punit(R, \Delta; \up \sigma I, \up \sigma \Gamma)\).
\item If \(\sigma \in \punit(R, \Delta; I, \Gamma)\) and \(g \in \unit(R, \Delta)\), then \([\sigma, g] = \up \sigma g\, g^{-1} \in \unit(I, \Gamma)\). In other words, \([\punit(R, \Delta; I, \Gamma), \unit(R, \Delta)] \leq \unit(I, \Gamma)\).
\item The embedding \(p_1^{-1} \colon \punit(R, \Delta; I, \Gamma) \rar \punit^{\mathrm{split}}(R \times_I R, \Delta \times_\Gamma \Delta), \sigma \mapsto \sigma \times \id\) is well-defined and the sequence
\[1 \rar \punit(R, \Delta; I, \Gamma) \xrightarrow{p_1^{-1}} \punit^{\mathrm{split}}(I \rtimes R, \Gamma \rtimes \Delta) \xrightarrow{p_2} \punit(R, \Delta) \rar 1\]
is short exact with the splitting \(d\).
\item There is a homomorphism \(\unit(R, \Delta) \rar \punit(R, \Delta)\) such that for all \(g \in \unit(R, \Delta)\), \(x \in R\), and \(u \in \Delta\) we have \(\up g x = \up{\alpha(g)} x\) and \(\up g u = (\gamma(g) \cdot \pi(u) \dotplus u) \cdot \inv{\alpha(g)}\). This homomorphism is stable under the action of \(\punit(R, \Delta)\) and maps \(\unit(I, \Gamma)\) into \(\punit(R, \Delta; I, \Gamma)\), image of any \(g \in \unit(R, \Delta)\) in \(\punit(R, \Delta)\) acts on \(\unit(R, \Delta)\) via the conjugation. Every element of \(\unit(R, \Delta)\) normalizes all subgroups \(\punit(R, \Delta; I, \Gamma) \leq \punit(R, \Delta)\). Finally, the image of \(\unit(I \rtimes R, \Gamma \rtimes \Delta)\) in \(\punit(I \rtimes R, \Gamma \rtimes \Delta)\) lies in \(\punit^{\mathrm{split}}(I \rtimes R, \Gamma \rtimes \Delta)\).
\item Suppose that we have odd form subalgebras \((S, \Theta) \subseteq (R, \Delta)\) and \((S', \Theta') \subseteq (R', \Delta')\), \((S, \Theta)\) has an orthogonal hyperbolic family of rank \(n \geq 3\) and is reduced. Then \(f \in \Hom(R, \Delta; R', \Delta')\) maps \((S, \Theta)\) into \((S', \Theta')\) (i.e. it induces an element of \(\Hom(S, \Theta; S', \Theta')\)) if and only if \(\up f {\eunit(S, \Theta)} \leq \unit(S', \Theta')\).
\end{enumerate}
\end{lemma}
\begin{proof}
The claims \((1)\), \((2)\), \((3)\), \((4)\) are trivial and \((5)\) may be checked through straightforward computations. For the last claim, let \(f \in \Hom(R, \Delta; R', \Delta')\) be such that \(\up f {\eunit(S, \Theta)} \leq \unit(S', \Theta')\). Then \(\up f (x - \inv x) \in S'\) for all \(x \in S_{i j}\), where \(i \neq \pm j\) and \(i, j \neq 0\). Since \(n \geq 3\), it follows that \(\up f S_{0'} \subseteq S'\). Similarly, \(\up f {(q_i \cdot x \dotplus q_{-j} \cdot (-\inv x) \dotminus \phi(x))} \in \Theta'\) for all \(x \in S_{i j}\), hence \(\up f {q_i} \in \Theta'\).

Now we will use ultrashort elementary transvections. For all \(i \neq 0\) and \(u \in \Theta^0_i\) we have
\[\up f {(\rho(u) + \pi(u) - \inv{\pi(u)})} \in S'\]
and
\[\up f {(u \dotminus \phi(\rho(u) + \pi(u)) \dotplus q_{-i} \cdot (\rho(u) - \inv{\pi(u)}))} \in \Theta'.\]
The first statement implies that \(\up f S \subseteq S'\) since \(S\) is generated by \(S_{0'}\) and \(\pi(\Theta^0_{0'})\) as an involution \(K\)-algebra. The second one implies that \(\up f {\Theta^0_{0'}} \subseteq \Theta'\), hence \(\up f \Theta \subseteq \Theta'\).
\end{proof}

One can also describe elements of \(\Hom(R, \Delta; S, \Theta)\) in terms of \(\Delta^0\) and \(\Theta^0\). Indeed, we have
\begin{lemma}\label{ProjectiveEquations}
Let \(f \in \Hom(R, \Delta; S, \Theta)\) and suppose that both \((R, \Delta)\) and \((S, \Theta)\) have orthogonal hyperbolic families (we will denote elements of these families as \(e_i\) and \(q_i\) for both odd form rings simultaneously). Let \(\delta(f) = (\sum_{i \neq 0}^\cdot \up f {q_i})^0 \in \Theta^0\) and \(\sub f u = (\up f u)^0 \in \Theta^0\) for all \(u \in \Delta^0\). Then
\begin{itemize}
\item \(\pi(\sub f u) = e_0 \up f {\pi(u)}\), \(\rho(\sub f u) = \up f {\inv{\pi(u)}} e_+ \up f {\pi(u)} + \up f \rho(u)\), \(\sub f {\phi(x)} = \phi(\up f x)\);
\item \(\sub f {(u \cdot x)} = \sub f u \cdot \up f x\), \(\sub f {(u \dotplus v)} = \sub f u \dotplus \sub f v \dotplus \phi(\up f {\inv{\pi(v)}} e_+ \up f {\pi(u)})\);
\item \(\pi(\delta(f)) = e_0 \up f{e_{0'}}\), \(\rho(\delta(f)) = \up f{e_{0'}} e_+ \up f {e_{0'}} - \up f {e_+}\), \(\delta(f) \cdot \up f {e_0} = \dot 0\).
\end{itemize}
Conversely, if there are a homomorphism of involution \(K\)-algebras \(R \rar S, x \mapsto \up f x\), an element \(\delta(f) \in \Theta^0\), and a map \(\Delta^0 \rar \Theta^0, u \mapsto \sub f u\) satisfying these identities, then there exists unique corresponding morphism \(f \colon (R, \Delta) \rar (S, \Theta)\).

The action of \(f\) on \(\unit(R, \Delta)\) in these terms is given by
\begin{align*}
\alpha(\up f g) &= \up f \alpha(g), \\
\gamma^0(\up f g) &= \delta(f) \cdot \up f \beta(g) \dotplus \sub f {(\gamma^0(g))} \dotplus \phi(\up f {\inv{\beta(g)}}\, \up f {e_0} e_+ \up f {e_{0'}} \up f {\beta(g)}), \\
\gamma^\circ(\up f g) &= \delta(f) \cdot \up f \beta(g) \dotplus \sub f {(\gamma^\circ(g))} \dotplus \phi((e_+ - \up f {e_+})\, \up f {\beta(g)} + \up f {\inv{\beta(g)}}\, \up f {e_0} e_+ \up f {e_{0'}} \up f {\beta(g)}).
\end{align*}

For \(g \in \unit(R, \Delta)\) we have
\begin{align*}
\delta(g) &= \gamma^0_{0'}(g) \cdot \inv{\alpha(g)} \dotplus \phi(\alpha(g) e_+ \beta(g) e_{0'} \inv{\alpha(g)}) = \gamma^\circ_{0'}(g) \cdot \inv{\alpha(g)} \\
\sub g u &= (\gamma^0(g) \cdot \pi(u) \dotplus u) \cdot \inv{\alpha(g)} = (\gamma^\circ(g) \cdot \pi(u) \dotplus u) \cdot \inv{\alpha(g)}.
\end{align*}

For \(f = \id\) we have \(\delta(\id) = \dot 0\) and \(\sub f u = u\). Finally, if \(g \colon (T, \Xi) \rar (R, \Delta)\) and \(f \colon (R, \Delta) \rar (S, \Theta)\) are morphisms and all three odd form algebras have orthogonal hyperbolic families, then
\begin{align*}
\delta(fg) &= \delta(f) \cdot \up{fg}{e_{0'}} \dotplus \sub f {\delta(g)} \dotplus \phi(\up{fg}{e_{0'}} \up f {e_0} e_+ \up f {e_{0'}} \up{fg}{e_{0'}}),\\
\sub{fg} u &= \delta(f) \cdot \up{fg}{\pi(u)} \dotplus \sub f{(\sub g u)} \dotplus \phi(\up{fg}{\inv{\pi(u)}}\, \up f {e_0} e_+ \up f {e_{0'}} \up {fg} {\pi(u)}).
\end{align*}
\end{lemma}
\begin{proof}
The relations on \(\delta(f)\) and \(u \mapsto \sub f u\) easily follow from the properties of \(u \mapsto u^0\). Clearly, \(f\) is uniquely determined by \(\delta(f)\) and \(u \mapsto \sub f u\), the relations exactly mean that \(f\) is a morphism.

The remaining formulas may be derived directly.
\end{proof}

Now we a ready to define general unitary groups as in \cite{OddStrucVor}. Let \((R, \Delta) \subseteq (T, \Xi)\) be an odd form \(K\)-subalgebra and \((I, \Gamma) \leqt (R, \Delta)\) be an odd form ideal. Then
\[\gunit(T, \Xi; R, \Delta) = \unit(T, \Xi)_{(R, \Delta)} = \{g \in \unit(T, \Xi) \mid \up g {(R, \Delta)} = (R, \Delta)\}\]
is the general unitary group. There is the obvious homomorphism \(\gunit(T, \Xi; R, \Delta) \rar \punit(R, \Delta)\), and the relative general unitary group \(\gunit(T, \Xi; R, \Delta; I, \Gamma)\) is the preimage of \(\punit(R, \Delta; I, \Gamma)\) under this homomorphism. Lemma \ref{ProjectiveUnitaryGroups} also holds for \(\gunit(T, \Xi; R, \Delta)\) instead of \(\punit(R, \Delta)\) and similarly for the relative groups. Also the sets of elementary transvections in \(\gunit(T, \Xi; R, \Delta; I, \Gamma)\), \(\unit(I, \Gamma)\), and \(\eunit(R, \Delta; I, \Gamma)\) coincide, if \((R, \Delta)\) has an orthogonal hyperbolic family of rank \(n \geq 3\) (this easily follows from the definition of the general unitary group).

The next proposition shows that, conversely, the projective unitary group may be reduced to the general unitary one (if we set \(G = \punit(R, \Delta)\)).

\begin{prop}\label{ProjectiveLifting}
Let \((R, \Delta)\) be an odd form \(K\)-algebra and \(G \rar \punit(R, \Delta)\) be a group homomorphism. Then there is a functorial odd form \(K\)-overalgebra \((R, \Delta) \subseteq (R[G], \Delta[G])\) such that \(G \rar \punit(R, \Delta)\) naturally lifts to a homomorphism \(G \rar \gunit(R[G], \Delta[G]; R, \Delta)\).
\end{prop}
\begin{proof}
Let \(R[G] = R \oplus \bigoplus_{1 \neq g \in G} (R \rtimes K) \beta(g)\) for formal symbols \(\beta(g)\) with the operations
\begin{align*}
x (y \beta(g)) &= (xy)\beta(g), \\
(x \beta(g)) y &= x\, \up g y \beta(g) + x\, \up g y - xy, \\
(x \beta(g)) (y \beta(h)) &= x\, \up g y \beta(gh) - x\, \up gy \beta(g) - x y \beta(h), \\
\inv{x \beta(g)} &= \beta(g^{-1}) \inv x = \up{g^{-1}}{\inv x} \beta(g^{-1}) + \up{g^{-1}}{\inv x} - \inv x.
\end{align*}
Then \(R[G]\) is an involution \(K\)-overalgebra of \(R\) and \(\up g x = \alpha(g) x \alpha(g)^{-1}\) for all \(x \in R\), \(g \in G\), where \(\alpha(g) = \beta(g) + 1 \in R[G] \rtimes K\). The group \(\Delta[G]\) is generated by symbols \([u]\), \(\phi(r)\), and \(\gamma(g) \cdot x\) for \(u \in \Delta\), \(x \in R \rtimes K\), \(r \in R[G]\), and \(g \in G\). Relations are the following:
\begin{itemize}
\item \(\phi(r + s) = \phi(r) \dotplus \phi(s)\), \(\phi(r) = 0\) for \(r \in \herm(R[G])\);
\item \([u \dotplus v] = [u] \dotplus [v]\), \(\phi(x) = [\phi(x)]\);
\item \(\gamma(1) = \dot 0\), \(\gamma(g) \cdot (x + y) = \gamma(g) \cdot x \dotplus \gamma(g) \cdot y \dotplus \phi(\inv y \inv{\beta(g)} x)\);
\item \(\phi(r) \dotplus [u] = [u] \dotplus \phi(r)\), \(\phi(r) \dotplus \gamma(h) \cdot y = \gamma(h) \cdot y \dotplus \phi(r)\);
\item \(\gamma(g) \cdot x \dotplus \gamma(h) \cdot y = \gamma(h) \cdot y \dotplus \gamma(g) \cdot x \dotplus \phi(-\inv x \beta(g^{-1}) \beta(h) y)\), \([u] \dotplus \gamma(g) \cdot x = \gamma(g) \cdot x \dotplus [u] \dotplus \phi(-\inv{\pi(u)} \beta(g) x)\).
\end{itemize}
It is easy to see that \(\phi(R[G]) \leqt \Delta[G]\) is a central subgroup and \(\Delta / \phi(\Delta) \oplus \bigoplus_{1 \neq g \in G} (R \rtimes K) \cong \Delta[G] / \phi(R[G])\) via \((u \dotplus \phi(\Delta)) \oplus \bigoplus_{1 \neq g} x_g \mapsto u \dotplus \sum^\cdot_{1 \neq g} \gamma(g) \cdot x_g \dotplus \phi(R[G])\). The operations are given by
\begin{itemize}
\item \(\pi([u]) = \pi(u)\), \(\pi(\gamma(g) \cdot x) = \beta(g) x\), \(\pi(\phi(r)) = 0\);
\item \(\rho([u]) = \rho(u)\), \(\rho(\gamma(g) \cdot x) = \inv x \inv{\beta(g)} x\), \(\rho(\phi(r)) = r - \inv r\);
\item \([u] \cdot x = [u \cdot x]\), \((\gamma(g) \cdot x) \cdot y = \gamma(g) \cdot xy\), \(\phi(r) \cdot x = \phi(\inv x r x)\);
\item \([u] \cdot \alpha(g) = \gamma(g) \cdot \up{g^{-1}}{\pi(u)} \dotplus [\up{g^{-1}} u]\), \((\gamma(g) \cdot x) \cdot \alpha(h) = \gamma(gh) \cdot \up{h^{-1}} x \dotminus \gamma(h) \cdot \up{h^{-1}} x\), \(\phi(r) \cdot \alpha(g) = \phi(\inv{\alpha(g)} r \alpha(g))\).
\end{itemize}
Then \((R[G], \Delta[G])\) is an odd form ring and \(\Delta \rar \Delta[G], u \mapsto [u]\) induces an embedding \((R, \Delta) \rar (R[G], \Delta[G])\). Moreover, \(G \rar \unit(R[G], \Delta[G]), g \mapsto (\beta(g), \gamma(g))\) is also an embedding lifting \(G \rar \punit(R, \Delta)\). Since its image lies in \(\gunit(R[G], \Delta[G]; R, \Delta)\), the proposition follows.
\end{proof}

\section{Examples}

In this section we will use all results from the rest of the paper, including theorem \ref{HomStability} at the end.

Now we consider odd form algebras arising from certain quadratic modules \((M, B, q)\) over rings \(R\) with \(\lambda\)-involutions and quadratic structures \(A\). By construction, all these algebras will be special unital with free orthogonal families of rank \(n\).

\begin{example}\label{LinearNoncommutativeOFA}
Let \(R\) be arbitrary ring and \(M = R^n\) be the free right module. Consider new ring \(\widetilde R = R^\op \times R\) with the involution \(\inv{(x^\op, y)} = (y^\op, x)\), the quadratic structure \(A = R\), and the operations \(x \cdot (y^\op, z) = yxz\), \(\varphi(x^\op, y) = x + y\), \(\tr(x) = (x^\op, x)\). Then \(M \oplus \Hom_R(M, R)\) becomes a quadratic module over \(\widetilde R\) with the multiplication \((m \oplus u) (x^\op, y) = my \oplus xu\) and the forms \(B(m_1 \oplus u_1, m_2 \oplus u_2) = ((u_2 m_1)^\op, u_1 m_2)\), \(q(m \oplus u) = u m\). This module is hyperbolic and its unitary group is isomorphic to \(\Aut_R(M) = \gl(n, R)\). The corresponding odd form ring is called linear odd form ring, it is \((T, \Xi)\), where \(T = \bigoplus_{i, j = 1}^n Re_{i j} \oplus \bigoplus_{i, j = -n}^{-1} R^\op e_{i j}\), \(\Xi^0_i = \dot 0\), \(e_0 = 0\), and \(e_i = e_{i i}\) for \(0 < i \leq |n|\) with the operations
\begin{align*}
(x e_{i j}) (y e_{k l}) &= 0 \text{ for } j \neq k, &
(r e_{i j}) (r' e_{j k}) &= rr' e_{i k} \text{ for } i, j, k > 0, \\
\inv{r e_{i j}} &= r^\op e_{-j, -i} \text{ for } i, j > 0, &
(r^\op e_{i j}) ({r'}^\op e_{j k}) &= (r'r)^\op e_{i k} \text{ for } i, j, k < 0.
\end{align*}
In this case \(\unit(T, \Xi) \cong \gl(n, R)\). If \(R\) is semi-local and \(n \geq 1\), then the orthogonal hyperbolic family is Morita complete (this family is always free).
\end{example}

If \((M, B, q)\) is arbitrary quadratic module over a ring \(R\) with a \(\lambda\)-involution and a quadratic structure \(A\), then the unitary group \(\unit(M, B, q)\) is naturally a subgroup in \(\gl(M)\). This embedding corresponds to an embedding of odd form rings. Indeed, \(T = \{(x^\op, y) \in \End_R(M)^\op \times \End_R(M) \mid B(xm, m') = B(m, ym')\}\) is an involution subring of \(\widetilde T = \End_R(M)^\op \times \End_R(M)\), and the corresponding special odd form parameter \(\Xi\) lies in \(\widetilde \Xi = \widetilde \Xi_{\mathrm{max}} \leq \Heis(\widetilde T, B_1)\). Clearly, \(\unit(\widetilde T, \widetilde \Xi) = \unit(M, 0, 0) = \gl(M)\). If \(M\) is free, then \((\widetilde T, \widetilde \Xi)\) is a linear odd form ring.

In the following examples we need to check that the projective unitary group coincides with the split classical projective group. Recall their definitions: \(\mathrm G_{\mathrm m}(-) \rightarrowtail \mathrm{GL}(n, -) \twoheadrightarrow \mathrm{PGL}(n, -)\), \(\mathrm G_{\mathrm m}(-) \rightarrowtail \mathrm{GSp}(2n, -) \twoheadrightarrow \mathrm{PGSp}(2n, -)\), and \(\mathrm G_{\mathrm m}(-) \rightarrowtail \mathrm{GO}(n, -) \twoheadrightarrow \mathrm{PGO}(n, -)\) are short exact sequences of group functors in the fppf-topology (even in the Zariski topology) for \(n \geq 1\). Here \(\mathrm G_{\mathrm m}(K) = K^*\) is diagonally embedded into \(\gl(n, K)\), \(\mathrm{GSp}(2n, K) = \{g \in \gl(2n, K) \mid \text{exists } \lambda \in K^* \text{ such that } B(gx, gy) = \lambda B(x, y)\}\) for a split symplectic module \(K^{2n}\), and \(\mathrm{GO}(n, K) = \{g \in \gl(n, K) \mid \text{exists } \lambda \in K^* \text{ such that } q(gx) = \lambda q(x)\}\) for a split quadratic module \(K^n\). Also, \(\mathrm{PGO}(2n + 1, K) \cong \mathrm{SO}(2n + 1, K)\) since \(\mathrm{GO}(2n + 1, K) = \mathrm{SO}(2n + 1, K) \times \mathrm G_{\mathrm m}(K)\). It is trivial to construct the sequence in all examples and to prove the exactness in all terms but the right one. The surjectivity in the sequence usually follows from the well-known fact that every automorphism of the matrix algebra over a commutative ring is inner locally in the Zariski topology. Equivalently, every automorphism of the algebra \(\mat(n, K)\) is inner, if \(K\) is a local commutative ring. We will prove the surjectivity using our theorem \ref{HomStability} since it also may be applied in the odd orthogonal case.

\begin{example}\label{LinearOFA}
For the linear odd form algebra of rank \(n \geq 1\) over a commutative ring \(K\) we have \(R = \bigoplus_{i, j = 1}^n Ke_{i j} \oplus \bigoplus_{i, j = -n}^{-1} Ke_{i j}\), \(\Delta^0_i = \dot 0\), \(e_0 = 0\), and \(e_i = e_{i i}\) for \(0 < |i| \leq n\) with the operations
\begin{align*}
e_{i j} e_{j k} &= e_{i k}, &
\inv{e_{i j}} &= e_{-j, -i}, &
e_{i j} e_{k l} &= 0 \text{ for } j \neq k.
\end{align*}
Clearly, the orthogonal hyperbolic family is free and Morita complete. Moreover, \(\unit(R, \Delta) \cong \gl(n, K)\) and \(\punit(R, \Delta) \cong \mathrm{PGL}(n, K) \rtimes (\mathbb Z / 2 \mathbb Z)(K)\), where \((\mathbb Z / 2 \mathbb Z)(K)\) is the group of idempotents of \(K\) with the operation \((e, f) \mapsto e + f - 2ef\). In order to prove that the obvious injective map \(\mathrm{PGL}(n, K) \rtimes (\mathbb Z / 2 \mathbb Z)(K) \rar \punit(R, \Delta)\) is also surjective, we need only to consider the case of local \(K\). Using theorem \ref{HomStability}, it suffices to prove surjectivity for \(n = 1\), but this case is obvious. Note for \(g \in \punit(R, \Delta)\) we have \(\delta(g) = \dot 0\) if and only if \(g \in \mathrm{PGL}(n, K)\).
\end{example}

Recall that a regular bilinear form \(B\) on a finitely generated projective \(K\)-module \(M\) is called symplectic if \(B(m, m) = 0\) for all \(m\). Locally in the Zariski topology any symplectic module is isomorphic to the split symplectic module \(M = K^{2n}\) with the form \(B(x, y) = \sum_{i > 0} (x_i y_{-i} - x_{-i} y_i)\) (we use the numeration \(-n, \ldots, -1, 1, \ldots, n\) for the coordinates on \(M\)).

\begin{example}\label{SymplecticOFA}
Let \(K\) be a commutative ring, \(n \geq 1\), \(M = K^{2n}\) be the split symplectic module of rank \(2n\), and \(B\) be the symplectic form. Then \(K\) has a trivial \((-1)\)-involution and a quadratic structure \(A = 0\), the zero map \(q \colon M \rar A\) is a quadratic form associated with \(B\). Applying our general construction of odd form algebras to \((M, B, q)\) we obtain the symplectic odd form algebra \((R, \Delta)\) of rank \(n\). More explicitly, \(R = \bigoplus_{0 < |i|, |j| \leq n} K e_{i j}\), \(\Delta^0_i = Ku_i\) for \(0 < |i| \leq n\), \(e_0 = 0\), and \(e_i = e_{i i}\) for \(0 < |i| \leq n\) with the operations
\begin{align*}
e_{i j} e_{k l} &= 0 \text{ for } j \neq k, &
\phi(x e_{-i, i}) &= 2 x u_i, &
xu_i \dotplus yu_i &= (x + y) u_i, \\
e_{i j} e_{j k} &= e_{i k}, &
\pi(x u_i) &= 0, &
(x u_i) \cdot (y e_{i, j}) &= \eps_i \eps_j xy^2 u_j, \\  
\inv{e_{i j}} &= \eps_i \eps_j e_{-j, -i}, &
\rho(x u_i) &= x e_{-i, i}.
\end{align*}
where \(\eps_i = 1\) for \(i > 0\) and \(\eps_i = -1\) for \(i < 0\). Here the orthogonal hyperbolic family is free and Morita complete. Also, \(\unit(R, \Delta) \cong \mathrm{Sp}(2n, K)\) is the split symplectic group, \(\punit(R, \Delta) \cong \mathrm{PGSp}(2n, K)\), and \(\gunit(T, \Xi; R, \Delta) \cong \mathrm{GSp}(2n, K)\), where \((T, \Xi)\) is the enveloping linear odd form algebra (of rank \(2n\)). This follows from theorem \ref{HomStability} and from easy case \(n = 1\) and local \(K\), where \(\mathrm{Sp}(2, K) = \mathrm{SL}(2, K)\) and \(\mathrm{GSp}(2, K) = \mathrm{GL}(2, K)\).
\end{example}

Now we consider classical quadratic forms. A map \(q \colon M \rar K\) from a finitely generated projective \(M\) over a commutative ring \(K\) is called quadratic form (in the classical sense) if \(q(mx) = q(m) x^2\) and \(B(x, y) = q(x + y) - q(x) - q(y)\) is bilinear. The form \(q\) is called regular of even rank if \(M\) is locally of even rank and \(B\) is regular (equivalently, the determinant of \(B\) is invertible). Locally in the \'etale topology every regular (classical) quadratic module \((M, q)\) of even rank is isomorphic to the split (classical) quadratic module \((K^{2n}, q)\) of rank \(2n\), where \(q(x) = \sum_{i > 0} x_i x_{-i}\) (again, we use the numeration \(-n, \ldots, -1, 1, \ldots, n\)).

\begin{example}\label{EvenOrthOFA}
Let \(K\) be a commutative ring, \(n \geq 1\), \(M = K^{2n}\) be the split (classical) quadratic module of rank \(2n\), \(q\) and \(B\) be the corresponding forms. Then \(K\) has a trivial \(1\)-involution and a quadratic structure \(A = K\) with the maps \(x \cdot y = xy^2\), \(\varphi(x) = x\), \(\tr(x) = 2x\). Hence \((M, B, q)\) is a quadratic module over \((K, A)\). Applying our general construction we obtain the even orthogonal odd form algebra \((R, \Delta)\) of rank \(n\). More explicitly, \(R = \bigoplus_{0 < |i|, |j| \leq n} K e_{i j}\), \(\Delta^0_i = \dot 0\) for \(0 < |i| \leq n\), \(e_0 = 0\), and \(e_i = e_{i i}\) for \(0 < |i| \leq n\) with the operations
\begin{align*}
\inv{e_{i j}} &= e_{-j, -i}, &
e_{i j} e_{k l} &= 0 \text{ for } j \neq k, &
e_{i j} e_{j k} &= e_{i k}.
\end{align*}
Here the orthogonal hyperbolic family is free and Morita complete. Also, \(\unit(R, \Delta) \cong \mathrm{O}(2n, K)\) is the split even orthogonal group, \(\punit(R, \Delta) \cong \mathrm{PGO}(2n, K)\), and \(\gunit(T, \Xi; R, \Delta) \cong \mathrm{GO}(2n, K)\), where \((T, \Xi)\) is the enveloping linear odd form algebra (of rank \(2n\)). This follows from theorem \ref{HomStability} and from the case \(n = 1\) and local \(K\), where \(\mathrm O(2, K) \cong K^* \rtimes \mathbb Z / 2 \mathbb Z\) and \(\mathrm{GO}(2, K) \cong (K^* \times K^*) \rtimes \mathbb Z / 2 \mathbb Z\).
\end{example}

It remains to consider classical quadratic modules of odd rank. Let \(M\) be a finitely generated projective module over a commutative ring, \(q \colon M \rar K\) be a classical quadratic form, and suppose that \(M\) is of constant odd rank. In this case the discriminant of \(B\) is divisible by \(2\) as an abstract polynomial with integer coefficients on entries of \(B\) (if \(M\) is free). Hence one can define the so-called half-discriminant of \(B\) over arbitrary ring \(K\). The form \(q\) is called semi-regular, if the half-discriminant of \(B\) is invertible. If \(2 \in K^*\), then semi-regularity is equivalent to regularity, but if \(2 = 0 \in K\), then \(B\) always satisfies \(B(x, x) = 0\), hence it cannot be regular. Locally in the \'etale topology every semi-regular (classical) quadratic module of odd rank is isomorphic to a module \(M = K^{2n + 1}\) with the form \(q(x) = \sum_{i > 0} x_i x_{-i} + a x_0^2\) for some \(a \in K^*\) (with the numeration \(-n, \ldots, n\) for coordinates on \(M\)). This element \(a\) coincides with the half-discriminant of \(B\) and can be made \(1\) locally in the fppf-topology.

\begin{example}\label{OddOrthOFA}
Let \(K\) be a commutative ring, \(n \geq 1\), \(M = K^{2n + 1}\) be a (classical) quadratic module with \(q(x) = \sum_{i > 0} x_i x_{-i} + x_0^2\). Then \(K\) has a trivial \(1\)-involution and a quadratic structure \(A = K\) with the maps \(x \cdot y = xy^2\), \(\varphi(x) = x\), \(\tr(x) = 2x\). Hence \((M, B, q)\) is a quadratic module over \((K, A)\). In principle we may apply our general construction and obtain some odd form algebra. This algebra even may be reduced using proposition \ref{OddFormRingReduction}, but the result will still be slightly cumbersome if \(2 \notin K^*\). Instead we define the odd orthogonal odd form algebra \((R, \Delta)\) of rank \(n\) directly as \(R = \bigoplus_{-n \leq i, j \leq n} K e_{i j}\), \(\Delta^0_i = Ku_i\) for \(-n \leq i \leq n\), \(e_i = e_{i i}\) for \(0 < |i| \leq n\) with the operations
\begin{align*}
e_{i j} e_{k l} &= 0 \text{ for } j \neq k, &
xu_i \cdot y e_{j k} &= \dot 0 \text{ for } i \neq j, &
\phi(xe_{-i, i}) &= \dot 0, \\
e_{i j} e_{j k} &= e_{i k} \text{ for } j \neq 0, &
xu_i \cdot ye_{i j} &= xyu_j \text{ for } i \neq 0, &
\pi(xu_i) &= xe_{0 i}, \\
e_{i 0} e_{0 j} &= 2e_{i j}, &
xu_0 \cdot ye_{0 i} &= 2xyu_i, &
\rho(xu_i) &= -x^2 e_{-i, i}, \\
\inv{e_{i j}} &= e_{-j, -i}, &
xu_i \dotplus yu_i &= (x + y) u_i. &
\end{align*}
In general \((R, \Delta)\) is special, but non-unital. If \(2 \in K^*\), then \(R \cong \mat(2n + 1, K)\). If \(K\) is a field of characteristic \(2\), then \(R\) is not semi-simple and its factor by the Jacobson radical is isomorphic to \(\mat(2n, K)\). Hence for arbitrary commutative \(K\) the orthogonal hyperbolic family is free and Morita complete. The homomorphism \(\mathrm{rep} \colon R \rar \End(M) = \mat(2n + 1, K)\) is given by \(e_{i j} \mapsto e_{i j}\) for \(j \neq 0\) and \(e_{i 0} \mapsto 2 e_{i 0}\). One can check that \(\unit(R, \Delta)\) acts on \(M\) by elements of \(\mathrm O(2n + 1, K)\). Explicit expressions for \(\unit(R, \Delta)\) and \(\punit(R, \Delta)\) are given in the next proposition.
\end{example}

\begin{prop}\label{OddOrthGroups}
If \((R, \Delta)\) is the odd orthogonal odd form algebra of rank \(n \geq 1\) over \(K\), then \(\unit(R, \Delta) \cong (\mathbb Z / 2 \mathbb Z)(K) \times \mathrm{SO}(2n + 1, K)\) and \(\punit(R, \Delta) \cong \mathrm{PGO}(2n + 1, K)\).
\end{prop}
\begin{proof}
The group \(\unit(R, \Delta)\) is given by the equations
\begin{align*}
\sum_{k \neq 0} \beta_{-k, -i} \beta_{k, j} + 2 \beta_{0, -i} \beta_{0, j} + \beta_{i, j} + \beta_{-j, -i} &= 0 \text{ for } i + j > 0, \\
\sum_{k > 0} \beta_{-k, i} \beta_{k, i} + \beta_{0, i}^2 + \beta_{-i, i} &= 0
\end{align*}
on the entries \(\{\beta_{i, j}\}_{i, j = -n - 1}^{n + 1}\). By the Jacobian criterion, this group scheme is smooth of relative dimension \(n(2n + 1)\) over \(K\) (these equations are transversal near the identity section by direct calculations, hence the scheme is smooth over every field and the equations are transversal near every point). Let \(Z(K) = \Ker(\unit(R, \Delta) \rar \punit(R, \Delta))\), then using lemma \ref{ProjectiveEquations} it is easy to see that \(Z(K) = \{g \in \unit(R, \Delta) \mid \beta(g) = e_{00} x + 2 \sum_{i \neq 0} xe_{ii}, x^2 + x = 0\} \cong (\mathbb Z / 2 \mathbb Z)(K)\). Now let \(\mathrm{Det}(x) = \det(\mathrm{rep}(x) + 1)\) for \(x \in R\), then \(\mathrm{Det}\) is a polynomial map and \(\mathrm{Det}(xy + x + y) = \mathrm{Det}(x) \mathrm{Det}(y)\). Also, \(1 - \mathrm{Det}(x)\) is always divisible by \(2\), hence there is unique polynomial map \(\mathrm D \colon R \rar K\) with integer coefficients such that \(1 - \mathrm{Det}(x) = 2 \mathrm D\). In particular, \(\mathrm D(xy + x + y) = \mathrm D(x) \mathrm D(y) + \mathrm D(x) + \mathrm D(y)\). Since \(\mathrm{Det}(\beta(g))^2 = 1\) for all \(g \in \unit(R, \Delta)\) and the affine ring of \(\unit(R, \Delta)\) is flat (hence torsion-free for \(K = \mathbb Z\)), we have \(\mathrm D(\unit(R, \Delta)) \leq (\mathbb Z / 2 \mathbb Z)(K)\). Clearly, \(\mathrm D\) induces an isomorphism between \(Z(K)\) and \((\mathrm Z / 2 \mathbb Z)(K)\). Let \(G(K) = \Ker(\mathrm D \colon \unit(R, \Delta) \rar (\mathbb Z / 2 \mathbb Z)(K))\).

It follows that \(\unit(R, \Delta) = Z(K) \times G(K)\). We have a homomorphism \(f \colon G(K) \rar \mathrm{SO}(2n + 1, K)\), and it is clearly injective. If for any \(g \in \mathrm{SO}(2n + 1, K)\) we will find a faithfully flat extension \(K' / K\) such that \(g \otimes 1\) lies in the image of \(G(K')\), then by faithfully flat descent \(g\) lies in the image of \(G(K)\). Note that \(\eunit(R, \Delta) \leq G(K)\) and \(f\) maps elementary transvections into elementary transvections (with the same parameter), hence by theorem \ref{SurjectiveKU} applied to \(\mathrm{O}(2n + 1, K)\) it suffices to consider the case \(n = 1\). We claim that every \(g \in \mathrm{SO}(3, K)\) lies in the elementary subgroup after a suitable faithfully flat base change. Indeed, we may assume that \(K\) is local, then \((g_{-1, 1}, 2g_{0 1}, g_{1 1})\) is unimodular and there is \(t \in K\) such that \((T_1(t) g)_{-1, 1}\) is invertible (where \(T_1(t)\) is the image of \(T_1(tu_1)\)). Then \((T_{-1}(s) T_1(t) g)_{1 1} = 1\) if and only if \(s\) satisfies some monic quadratic equation, hence after a faithfully flat base change we may assume that \(g_{1 1} = 1\). Multiplying by elementary transvections as in the proof of theorem \ref{SurjectiveKU}, we may further assume that \(g\) may differ from the identity matrix only in the middle column. Since \(g \in \mathrm{SO}(3, K)\), it follows that \(g = 1\) and the claim is proved. In other words, \(\unit(R, \Delta) \cong (\mathbb Z / 2 \mathbb Z)(K) \times \mathrm{SO}(2n + 1, K)\).

It remains to prove that \(G(K) \rar \punit(R, \Delta)\) is an isomorphism. This map is injective, hence it suffices to prove that every \(\sigma \in \punit(R, \Delta)\) lies in the image after suitable faithfully flat base change, hence we may assume that \(K\) is local. By theorem \ref{HomStability} it suffices to consider only the case \(n = 1\). This can be done similarly to lemma \ref{UnitaryExtension} and proposition \ref{HomDecomposition}. There is \(g \in \unit(\up \sigma {e_1}, e_1; R, \Delta)\), \(x = \alpha(g)\) is unimodular, hence as above after a base change and multiplication by an element of \(\eunit(R, \Delta)\) we may assume that \(x = e_1\). In other words, \(\up \sigma{e_1} = e_1\) and \(\up \sigma{e_{-1}} = e_{-1}\). It follows that \(\up{\sigma}{e_{i j}} = c_{i j} e_{i j}\) for some \(c_{i j} \in K^*\). But then the relations on \(e_{i j}\) and \(u_i\) show that \(\sigma\) lies in the image of \(\mathrm{SO}(2, K) \leq \mathrm{SO}(3, K)\).
\end{proof}

The remaining examples show various properties of general odd form rings. Non-special odd form algebras may be used to describe affine groups:

\begin{example}\label{Affine}
Let \((R, \Delta)\) be an odd form \(K\)-algebra and \(X\) be a right \((R \rtimes K)\)-module. Then \((R, \Delta \times X)\) is also an odd form \(K\)-algebra with the operations \(\pi(x) = \rho(x) = 0\) and \(x \cdot r = xr\) for \(x \in X\) and \(r \in R \rtimes K\). There is canonical isomorphism \(\unit(R, \Delta \times X) \cong X \rtimes \unit(R, \Delta)\) with the left action \(gx = x\inv{\alpha(g)}\) of \(\unit(R, \Delta)\) on \(X\). In particular, \(K^n \rtimes \gl(n, K)\), \(K^{2n} \rtimes \mathrm{Sp}(2n, K)\), and \(K^n \rtimes \mathrm O(n, K)\) are all unitary groups for appropriate non-special odd form algebras. Clearly, the embedding \((R, \Delta) \subseteq (R, \Delta \times X)\) preserves freeness and Morita completeness of orthogonal hyperbolic families.
\end{example}

Also, stable unitary groups now are also unitary groups of certain non-unital odd form algebras.

\begin{example}\label{StableOFA}
Let \((R[n], \Delta[n])\) be a family of odd form algebras with free orthogonal hyperbolic families of ranks \(n\) and suppose that \((R[n - 1], \Delta[n - 1]) = (R[n]_{|n|'}, \Delta_{|n|'}^{|n|'})\) for all \(n\) (so their hyperbolic pairs with indices \(1, \ldots, n - 1\) coincide). Then \((R, \Delta) = \varinjlim_n (R[n], \Delta[n])\) is an odd form algebra with infinite free orthogonal hyperbolic family, it unitary group is the direct limit of \(\unit(R[n], \Delta[n])\).
\end{example}

Finally, we give two counter-examples.

\begin{example}\label{SmallElementary}
In general the elementary unitary groups depends on the orthogonal hyperbolic family. Let \((R, \Delta) = (R_1, \Delta_1) \times (R_2, \Delta_2)\) be the product of linear odd form \(\mathbb Q\)-algebras of ranks \(n, m \geq 2\). Using hyperbolic pairs from different factors we obtain different elementary unitary groups \(\mathrm E(n, \mathbb Q) = \mathrm{SL}(n, \mathbb Q)\) and \(\mathrm E(m, \mathbb Q) = \mathrm{SL}(m, \mathbb Q)\). Clearly, the smallest of these families is not Morita complete.
\end{example}

\begin{example}\label{Twisting}
Even in the case of linear odd form algebra over a commutative ring it is not true that the map \(\unit(R, \Delta) \rar \punit(R, \Delta)\) is surjective. Let \(K = \mathbb Z[i \sqrt 5]\) and \(\mathfrak a = (2, 1 + i \sqrt 5) \leqt K\), then \(\mathfrak a^2 = (2)\) and \(\mathfrak a\) is a non-principal prime ideal (it is an element of order \(2\) in \(\mathrm{Pic}(K)\), the ring \(K\) is a Dedekind domain). The conjugation by the matrix \(g = \sMat 2 {1 + i \sqrt 5} {1 - i \sqrt 5} 2\) is an automorphism of \(\mat(2, K)\) since \(\frac g {\sqrt 2} \in \gl(2, K[\sqrt 2])\) and \(K\) is integrally closed. It is easy to see that \(\up g {e_{1 1}} {\mat(2, K)} e_{1 1}\) is isomorphic to \(\mathfrak a\) as a \(K\)-module, hence this automorphism is outer. Moreover, there is an automorphism \(\sigma\) on \(\mat(n, K)\) with \(\up \sigma {e_{1 1}} {\mat(n, K)} e_{1 1} \cong \mathfrak a\) if and only if \(n\) is even, since \((\up \sigma {e_{1 1}} {\mat(n, K)} e_{1 1})^{\oplus n} \cong \mat(n, K) e_{1 1} \cong K^{\oplus n}\) and \(\mathfrak b^{\oplus n} \cong K^{\oplus (n - 1)} \oplus \mathfrak b^{\otimes n}\) for any fraction ideal \(\mathfrak b\) over Dedekind domain.
\end{example}

\section{Stable rank}

Until the end of this section \((R, \Delta)\) is an odd form ring with a free orthogonal hyperbolic family of rank \(n \geq 1\). In the unitary group \(\unit(n; R, \Delta) = \unit(R, \Delta)\) there is a subgroup
\begin{align*}
\unit(n - 1; R, \Delta) &= \{g \in \unit(n; R, \Delta) \mid \beta_{*, \pm n}(g) = \beta_{\pm n, *}(g) = 0, \gamma_{\pm n}^\circ(g) = \dot 0\} =\\
&= \{g \in \unit(n; R, \Delta) \mid \beta_{*, \pm n}(g) = \beta_{\pm n, *}(g) = 0, \gamma_{\pm n}^0(g) = \dot 0\} =\\
&= \{g \in \unit(n; R, \Delta) \mid \beta_{*, \pm n}(g) = \beta_{\pm n, *}(g) = 0, \gamma_{\pm n}(g) = \dot 0\}.
\end{align*}
Clearly, \(\unit(n - 1; R, \Delta) \cong \unit(R_{|n|'}, \Delta_{|n|'}^{|n|'}) \cong \unit(e_{|n|'}, e_{|n|'}; R, \Delta)\). The elementary subgroup \(\eunit(n - 1; R, \Delta) \leq \unit(n - 1; R, \Delta)\) is generated by all transvections \(T_{i j}(x), T_j(u) \in \eunit(R, \Delta)\) with \(i, j \neq \pm n\). If \(n \geq 2\), then we may define similarly \(\unit(n - 2; R, \Delta)\) and so on.

Conversely, if the orthogonal hyperbolic family is free, then we may define \(\unit(n + 1; R, \Delta)\). Indeed, in this case \((R, \Delta)\) may be embedded into an odd form algebra \((\widetilde R, \widetilde \Delta)\) with a free orthogonal hyperbolic family of rank \(n + 1\) (unique up to the canonical isomorphism) such that \(R = \widetilde R_{|n + 1|'}\) and \(\Delta = \widetilde \Delta_{|n + 1|'}^{|n + 1|'}\). In this case we set \(\unit(n + 1; R, \Delta) = \unit(n + 1; \widetilde R, \widetilde \Delta) = \unit(\widetilde R, \widetilde \Delta)\).

In order to prove stabilization for \(\kunit\), we will use the standard approach with stable ranks and \(\Lambda\)-stable ranks. First of all, let \((R^{\mathrm{ev}}, \Delta^{\mathrm{ev}}) = (R_{0'}, \rho(\Ker(\pi|_{\Delta_{0'}})))\) be the even part of \((R, \Delta)\), then \(R^{\mathrm{ev}}\) is isomorphic to the matrix algebra \(\mat(n, R_{|1|})\) and \(\Delta^{\mathrm{ev}} \leq R^{\mathrm{ev}}\) is a form parameter (in the classical sense of Bak). Note that \(\Delta^{\mathrm{ev}}\) is also a special odd form parameter with the maps \(\rho^{\mathrm{ev}}(u) = u\), \(\pi^{\mathrm{ev}}(u) = 0\), and \(\phi^{\mathrm{ev}}(x) = x - \inv x\). There is a canonical bijection between form parameters on \(R_{|1|}\) and on \(\mat(n, R_{|1|})\) since these involution rings are Morita equivalent in the bicategory of hermitian bimodules, see \cite{OddStrucVor}. Explicitly, \(\Lambda \leq R_{|1|}\) corresponds to \(\bigoplus_{1 \leq i \leq n} e_{|i|, |1|} \Lambda e_{|1|, |i|} \oplus \bigoplus_{1 \leq i < j \leq n} \{x - \inv x \mid x \in R_{|i|, |j|}\}\). Every elementary transvection in \(\unit(R^{\mathrm{ev}}, \Delta^{\mathrm{ev}})\) may be lifted into \(\unit(R, \Delta)\).

Recall that the condition \(\mathrm{sr}(R_1) \leq k - 1\) means that for any right unimodular sequence \(x_1, \ldots, x_k \in R_1\) of length \(k\) (i.e. if there exist \(y_i \in R_1\) with \(\sum_i y_i x_i = 1\)) there are \(a_1, \ldots, a_{k - 1} \in R_1\) such that \(x_1 + a_1 x_k, \ldots, x_{k - 1} + a_{k - 1} x_k\) is right unimodular of length \(k - 1\).

If there are \(e_{-1, 1} \in R_{-1, 1}\) and \(e_{1, -1} \in R_{1, -1}\) such that \(e_{-1, 1} e_{1, -1} = e_{-1}\) and \(e_{1, -1} e_{-1, 1} = e_1\), then we may use the standard definition of \(\Lambda\)-stable rank in our situation. In general we say that \(\Lambda \mathrm{sr}(\eta_1; R, \Delta) \leq k - 1\) if \(\mathrm{sr}(R_1) \leq k - 1\) and for any right unimodular \(x_{-k}, \ldots, x_{-1}, x_1, \ldots, x_k\) with \(x_i \in R_1\) for \(i > 0\) and \(x_i \in R_{-1, 1}\) for \(i < 0\) (i.e. there are \(y_{-k}, \ldots, y_{-1}, y_1, \ldots, y_k\) such that \(\sum_i y_i x_i = e_1\)) there exists a matrix \(\{a_{i j} \in R_{1, -1}\}_{i = 1, j = -k}^{k, -1}\) such that \(a_{i j} = -\inv{a_{-j, -i}}\), \(a_{i, -i} \in \Delta^{\mathrm{ev}}\), and the sequence \(x_1 + \sum_{i = -k}^{-1} a_{1 i} x_i, \ldots, x_k + \sum_{i = -k}^{-1} a_{k i} x_i\) in \(R_1\) is right unimodular.

We need some basic properties of stable ranks and \(\Lambda\)-stable ranks. For ordinary stable ranks all proofs may be found in Bass's book \cite{Bass}, hence they will be omitted. In paper \cite{UnitStabBak1} all properties of the usual \(\Lambda\)-stable rank were proved, though we still have to modify them using our definition. Let \(\unit^-(R^{\mathrm{ev}}, \Delta^{\mathrm{ev}}) = \{g \in \unit(R^{\mathrm{ev}}, \Delta^{\mathrm{ev}}) \mid \alpha_{i j}(g) = 0 \text{ for } i < 0 < j\}\), it is isomorphic to a semi-direct product of \(\gl(n, R_1)\) and the abelian group \(\Delta^{\mathrm{ev}}_{+, -}\).

\begin{lemma}\label{EquivalentLSR}
The condition \(\Lambda\mathrm{sr}(\eta_1; R, \Delta) \leq n - 1\) is equivalent modulo \(\mathrm{sr}(R_1) \leq n - 1\) to the following: for every unimodular \(x \in R^{\mathrm{ev}}_{* 1}\) there exists \(g \in \unit^-(R^{\mathrm{ev}}, \Delta^{\mathrm{ev}})\) such that \(e_+gx\) is unimodular.
\end{lemma}
\begin{proof}
This is clear, since multiplication by elements from \(\gl(n, R_1)\) preserves unimodularity.
\end{proof}

\begin{lemma}\label{IncreasingSR}
If \(\mathrm{sr}(R_1) \leq n - 2\), then \(\mathrm{sr}(R_1) \leq n - 1\). Moreover, suppose that \(\mathrm{sr}(R_1) \leq m \leq n - 1\) and a sequence \(x_1, \ldots, x_n \in R_1\) is unimodular. Then there are \(a_1, \ldots, a_m \in R_1\) such that \(x_1 + a_1 x_n, \ldots, x_m + a_m x_n, x_{m + 1}, \ldots, x_{n - 1}\) is unimodular.
\end{lemma}
\begin{proof}
Omitted.
\end{proof}

The next lemma shows that the number \(\Lambda\mathrm{sr}(\eta_1; R, \Delta)\) in \(\{0, 1, \ldots, \infty\}\) is well-defined (for \(\mathrm{sr}(R_1)\) this follows from the previous lemma).

\begin{lemma}\label{IncreasingLSR}
If \(\Lambda\mathrm{sr}(\eta_1; R, \Delta) \leq n - 2\), then \(\Lambda\mathrm{sr}(\eta_1; R, \Delta) \leq n - 1\).
\end{lemma}
\begin{proof}
Suppose that \(\Lambda\mathrm{sr}(\eta_1; R, \Delta) \leq n - 2\) and \(x \in R^{\mathrm{ev}}_{* 1}\) is unimodular, so there is \(y \in R^{\mathrm{ev}}_{1 *}\) such that \(yx = e_1\). In particular, the sequence \(y e_{-n} x, \ldots, y e_{-1} x, e_1 x, \ldots, e_{1 n} x\) in \(R_1\) is unimodular. By lemma \ref{IncreasingSR} there are \(z_1, \ldots, z_{n - 1}, z'_1, \ldots, z'_{n - 1} \in R_1\) such that \(e_{-1, 1 - n} x, \ldots, e_{-1} x, e_1 x + z_1 e_{1 n} x + z'_1 y e_{-n} x, \ldots, e_{1, n - 1} x + z_{n - 1} e_{1 n} x + z'_{n - 1} y e_{-n} x\) is unimodular. But now there is \(a \in (e_+ - e_n) \Delta^{\mathrm{ev}} (e_- - e_{-n})\) such that \(e_{-n} x, e_1 x + \sum_{i = -1}^{-n - 1} e_1 a e_i x, \ldots, e_{1, n - 1} x + \sum_{i = -1}^{-n - 1} e_{1, n - 1} a e_i x, e_n x\) is also unimodular. It follows that \(e_{-1, -n} x, \ldots, e_{-1, -2} x, e_{1 2} x + \sum_{i = -1}^{-n - 1} e_{12} a e_i x + w_2 X_1, \ldots, e_{1, n - 1} x + \sum_{i = -1}^{-n - 1} e_{1, n - 1} a e_i x + w_{n - 1} X_1, e_{1 n} x + w_n X_1\) is unimodular for suitable \(w_2, \ldots, w_n \in R_1\), where \(X_1 = e_1 x + \sum_{i = -1}^{-n - 1} e_1 a e_i x\). Applying our \(\Lambda\)-stable rank condition once again, we obtain the desired result.
\end{proof}

\begin{lemma}\label{FactorLSR}
If \((I, \Gamma) \leqt (R, \Delta)\) is an odd form ideal, then \(\mathrm{sr}((R / I)_1) \leq \mathrm{sr}(R_1)\) and \(\Lambda\mathrm{sr}(\eta_1; R / I, \Delta / \Gamma) \leq \Lambda\mathrm{sr}(\eta_1; R, \Delta)\).
\end{lemma}
\begin{proof}
It suffices to show that any unimodular sequence in the factor-ring may be lifted into \(R_{|1|}\). If \(x \in (R / I)^{\mathrm{ev}}_{* 1}\) is unimodular, then in certainly may be lifted to \(\widetilde x \in R^{\mathrm{ev}}_{* 1}\) such that there is \(\widetilde y \in R^{\mathrm{ev}}_{1 *}\) with the property \(e_1 - \widetilde y \widetilde x \in I_1\). Let \(z = 1 - \widetilde y \widetilde x \in I_1\), then \(\widetilde y e_{-n} \widetilde x, \ldots, \widetilde y e_{-1} \widetilde x, e_1 \widetilde x, \ldots, e_n \widetilde x, z\) is unimodular sequence. By lemma \ref{IncreasingSR} we can modify \(e_+ \widetilde x\) in such a way that the new sequence \(\widetilde x'\) is still a lifting of \(x\) and is unimodular.
\end{proof}

\begin{lemma}\label{SemilocalLSR}
If \((R, \Lambda)\) is semi-local, then \(\Lambda \mathrm{sr}(\eta_1; R, \Delta) \leq 1\) (it equals \(0\) if and only if \(R_1 = 0\)).
\end{lemma}
\begin{proof}
We may assume that \((R, \Delta)\) has a free orthogonal hyperbolic family of rank \(n \geq 2\) (actually, we will prove that \(\Lambda\mathrm{sr}(\eta_1; R, \Delta) \leq n - 1\) for all such \(n\)). Let \(x \in R^{\mathrm{ev}}_{* 1}\) be a unimodular sequence. We will prove that there is \(g \in \langle T_{i j}(R_{i j}) \mid ij > 0 \text{ or } i > 0 > j; i \neq \pm j \rangle\) such that \(e_+ gx\) is unimodular, and this clearly may be checked modulo the Jacobson radical. Hence we can assume that \(R = R^{\mathrm{ev}}\), \(\Delta = \Delta^{\mathrm{ev}}\), and \(R\) is a semi-simple ring. Since our property may be checked on each simple factor of \(R\) separately, without loss of generality \(R\) is simple, \(R_1 \cong \mat(k, D)\) for \(k > 0\) and some division ring \(D\).

Clearly, there is such \(g\) that \(e_+ gx\) has the maximal possible rank as a matrix over \(D\). If this rank equals \(k\), there is nothing to prove. Suppose that this is not the case. By Gauss's elimination, there is \(h \in \langle T_{i j}(R_{i j}) \mid i, j > 0 \rangle\) such that \(e_1 hgx = e_+ hgx\). By the same argument, there is \(h' \in \langle T_{i 2}(R_{i 2}) \mid 0 < i \neq 2 \rangle\) such that \((e_{-2} + e_+) h'hgx\) has rank larger than the one of \(e_+ h'hgx = e_+ hgx\). Since \(e_1 h'hgx\) is not invertible, there is an idempotent \(0 \neq t \in R_1\) such that \(t e_1 h'hgx = 0\). Also, there is \(s \in R_{1, -2}\) such that \(ts e_{-2} h'hgx \notin R_1 e_1 h'hgx\). Hence \(e_+ T_{1, -2}(ts) h'hgx\) has the rank larger than the one of \(e_+ x\), a contradiction.
\end{proof}

\begin{lemma}\label{LimitLSR}
If \((R, \Delta) = \varinjlim_i (R_i, \Delta_i)\) is a direct limit of odd form rings with a common free orthogonal hyperbolic family, then \(\mathrm{sr}(R_1) \leq \liminf_i \mathrm{sr}((R_i)_1)\) and \(\Lambda\mathrm{sr}(\eta_1; R, \Delta) \leq \liminf_i \Lambda\mathrm{sr}(\eta_1; R_i, \Delta_i)\).
\end{lemma}
\begin{proof}
Let \((R_{i_s}, \Delta_{i_s})\) be a cofinal family such that
\[\Lambda\mathrm{sr}(\eta_1; R_{i_s}, \Delta_{i_s}) \leq \liminf_i \Lambda\mathrm{sr}(\eta_1; R_i, \Delta_i)\]
(if the lower limit is finite, then we may choose such a family that equalities hold). Now every unimodular sequence in \(R_{* 1}\) comes from some \((R_{i_s})_{* 1}\), hence \(\Lambda \mathrm{sr}(\eta_1; R_{i_s}, \Delta_{i_s}) \leq \liminf_i \Lambda\mathrm{sr}(\eta_1; R_i, \Delta_i)\) implies the required condition on this sequence.
\end{proof}

\section{Stable rank and dimension}

We say that a commutative ring \(K\) has Bass --- Serre dimension at most \(d\) and write \(\mathrm{BS}(K) \leq d\) if the space \(\mathrm{Max}(K)\) of maximal ideals of \(K\) with the Zariski topology may be decomposed into a finite union of irreducible Noetherian subspaces of dimension at most \(d\) (for example, it is true if \(\dim(K) \leq d\), where \(\dim(K)\) is the Krull dimension). Condition \(\mathrm{BS}(K) \leq 0\) is equivalent to the semi-locality of \(K\). We also set \(\mathrm{BS}(0) = -\infty\).

In this section we are going to bound \(\Lambda\mathrm{sr}(\eta_1; R, \Delta)\) by the Bass --- Serre dimension of \(K\) if \((R, \Delta)\) is quasi-finite. The letter \(d\) also means the bound on the Bass --- Serre dimension from the above, it takes values in \(\{-\infty, 0, 1, \ldots, +\infty\}\). Also, \(d + 1\) and \(d + 2\) mean ordinary sums if \(d \geq 0\), but they mean \(0\) and \(1\) correspondingly if \(d = -\infty\) (in other words, the dimension \(-\infty\) behaves like \(-1\) under the increment). Let us start with a couple of algebraic lemmas.

\begin{lemma}\label{UnimodularityLocus}
Let \((R, \Delta)\) be a finite odd form \(K\)-algebra with an orthogonal hyperbolic family of rank \(n \geq 1\) and \(x \in R_{0' 1}\). Then there is a Zariski open subset \(U \subseteq \mathrm{Spec}(K)\) such that for every commutative \(K\)-algebra \(K'\) the element \(x\) is right unimodular in \(R \otimes K'\) if and only if the scheme morphism \(\mathrm{Spec}(K') \rar \mathrm{Spec}(K)\) factors through \(U\).
\end{lemma}
\begin{proof}
Note that \(x\) is unimodular in \(R\) if and only if the homomorphism of \(K\)-modules \(R_{1 0'} \rar R_1, y \mapsto yx\) is surjective. Now the lemma follows from the general scheme-theoretical result: if \(S\) is any scheme, \(\mathcal F\) and \(\mathcal G\) are quasi-coherent sheaves of \(\mathcal O_S\)-modules, \(f \colon \mathcal F \rar \mathcal G\) is a module sheaf morphism, and \(\mathcal G\) is of finite type, then there is a Zariski open subset \(U \subseteq S\) such that for every scheme morphism \(\varphi \colon T \rar S\) the module sheaf morphism \(\varphi^*(f) \colon \varphi^*(\mathcal F) \rar \varphi^*(\mathcal G)\) is an epimorphism if and only if \(\varphi\) factors through \(U\).
\end{proof}

\begin{lemma}\label{AdvancedLifting}
Let \((R, \Delta)\) be a finite odd form \(K\)-algebra with an orthogonal hyperbolic family of rank \(n \geq 1\), \(x \in R_{0' 1}\), \(\eps = \sum_{i \in I} e_i\) for some \(I \subseteq \{-n, \ldots, -1, 1, \ldots, n\}\), \(\mathfrak p \leqt K\) be a prime ideal, \(\mathfrak a \not \leq \mathfrak p\) be an ideal of \(K\). Suppose that \(\eps g x\) is right unimodular in \(R \otimes \kappa(\mathfrak p)\) for some \(g \in \eunit(R^{\mathrm{ev}} \otimes \kappa(\mathfrak p), \Delta^{\mathrm{ev}} \otimes \kappa(\mathfrak p))\). Then there is \(\widetilde g \in \eunit(R^{\mathrm{ev}}, \Delta^{\mathrm{ev}})\) such that \(\widetilde g\) and \(g\) are products of elementary transvections with the same indices in the same order, \(\beta(\widetilde g) \in R \mathfrak a\), and \(\eps \widetilde g x\) is right unimodular in \(R \otimes \kappa(\mathfrak p)\).
\end{lemma}
\begin{proof}
Let \(g = \prod_{s = 1}^N g_s\) be the decomposition into a product of elementary transvections. If the residue field \(\kappa(\mathfrak p)\) is finite (or, more generally, if \(\mathfrak p\) is maximal), then each \(g_s\) is the image of some elementary transvection \(\widetilde g_s \in \eunit(R^{\mathrm{ev}}, \Delta^{\mathrm{ev}}; R^{\mathrm{ev}} \mathfrak a, \Delta^{\mathrm{ev}} \cdot \mathfrak a)\) with the same indices, hence we are done. Hence we may assume that the residue field is infinite.

We define new element \(g(k)\) that depends polynomially on \(k \in \kappa(\mathfrak p)\) (i.e. \(\alpha(g(T)) \in R \otimes \kappa(\mathfrak p)[T]\) for an indeterminate \(T\)) with the property \(g(1) = g\). Let \(g(k) = \prod_{s = 1}^N g_s(k)\) and for every \(s\) let \(g_s(k) = T_{i j}(y k)\) if \(g_s = T_{i j}(y)\) and \(g_s(k) = T_i(u \cdot k)\) if \(g_s = T_i(u)\). Note that by lemma \ref{UnimodularityLocus} the element \(\eps g(k) x\) is unimodular for all but a finite number of \(k\). There is \(d \in K \setminus \mathfrak p\) such that for every \(k\) in the image of the ideal \((d)\) in \(\kappa(\mathfrak p)\) all transvections \(g_s(k)\) may be lifted into \(\eunit(R^{\mathrm{ev}}, \Delta^{\mathrm{ev}})\). Now the image of \(\mathfrak ad\) in \(\kappa(\mathfrak p)\) is infinite, hence there is \(k\) in this image such that \(\eps g(k) x\) is unimodular. Let \(\widetilde g_s\) be the liftings of \(g_s(k)\) in \(\eunit(R^{\mathrm{ev}}, \Delta^{\mathrm{ev}}; R^{\mathrm{ev}} \mathfrak a, \Delta^{\mathrm{ev}} \cdot \mathfrak a)\) that are elementary transvections with the same indices. It follows that \(\widetilde g = \prod_{s = 1}^N \widetilde g_s\) satisfies all conditions.
\end{proof}

We will also need the following technical lemma.

\begin{lemma}\label{DimensionSR}
Let \((R, \Delta)\) be a finite odd form \(K\)-algebra with a free orthogonal hyperbolic family of rank \(n \geq d + 2\), \(x \in R_{0' 1}\). Suppose that \(\mathfrak p_1, \ldots, \mathfrak p_N \leqt K\) are prime ideals and \(e_{1 n} x\) is invertible in \(R_1 \otimes \kappa(\mathfrak p_s)\) for all \(s\). Let also \(\mathfrak a \leqt K\) be such that \(\mathrm{BS}(K / \mathfrak a) \leq d\), \(\mathfrak a \not \leq \mathfrak p_s\) for all \(s\), and \(x\) is right unimodular in \(R \otimes K / \mathfrak a\). Then there is \(g \in \langle T_{i, -n}(R_{i, -n}), T_{i j}(R_{i j}) \mid 1 \leq i, j \leq d + 1; i \neq j \rangle\) such that \(e_{1 n} gx\) is still invertible in \(R_1 \otimes \kappa(\mathfrak p_s)\) for all \(s\) and \(e_{(-n)'} gx\) is right unimodular in \(R \otimes K / \mathfrak a\).
\end{lemma}
\begin{proof}
The proof is by induction on \(d\), the case \(d = -\infty\) is obvious. Let \(\mathrm{Max}(K / \mathfrak a) = \bigcup_{t = 1}^M X_t\) be the decomposition into a finite union of irreducible Noetherian subspaces of dimensions at most \(d\). The set \(\{\bigcap_{\mathfrak m \in X_t} \mathfrak m \mid 1 \leq t \leq M\}\) consists of prime ideals of \(K / \mathfrak a\) (or, equivalently, of prime ideals of \(K\) containing \(\mathfrak a\)). Let \(\mathfrak q_1, \ldots, \mathfrak q_{M'}\) be all elements of this set without repetitions is such an order that \(\mathfrak q_{t_1} \not \leq \mathfrak q_{t_2}\) for \(t_1 < t_2\) (it is possible that \(M' < M\)). Let also \(e_{[1d]} = e_1 + \ldots + e_d\) and \(e_{[1d]'} = 1 - e_{[1d]}\). By lemmas \ref{SemilocalLSR} and \ref{AdvancedLifting} there are \(g_1, \ldots, g_{M'} \in \langle T_{d + 1, i}(R_{d + 1, i}) \mid 1 \leq i \leq d \rangle\) such that \(e_{[1d]'} g_t \cdots g_1 x\) is unimodular in \(R \otimes \kappa(\mathfrak q_t)\) and \(\beta(g_t) \in R \mathfrak q_1 \cdots \mathfrak q_{t - 1}\) for all \(t\). Hence without loss of generality we may assume that \(e_{[1d]'} x\) is unimodular in \(R \otimes \kappa(\mathfrak q_t)\) for all \(t\). Note that \(\mathfrak q_t \not \leq \mathfrak p_s\) for all \(t\) and \(s\), because \(\mathfrak a \leq \mathfrak q_t\).

We are going to find elements \(g_1, \ldots, g_{M'} \in T_{d + 1, -n}(R_{d + 1, -n})\) such that \(\beta(g_t) \in R \mathfrak q_1 \cdots \mathfrak q_{t - 1}\), \(e_{(-n)'} e_{[1d]'} g_t \cdots g_1 x\) are unimodular in \(R \otimes \kappa(\mathfrak q_t)\), and \(e_{1 n} g_t \cdots g_1 x\) are still invertible in \(R_1 \otimes \kappa(\mathfrak p_s)\) for all \(s, t\). Suppose that we already have \(g_1, \ldots, g_{t - 1}\). Without loss of generality \(\mathfrak p_s\) are distinct, \(\mathfrak p_{s_1} \not \leq \mathfrak p_{s_2}\) for all \(s_1 < s_2\), and \(\mathfrak p_s \leq \mathfrak q_t\) if and only if \(s > s_0\). By lemmas \ref{SemilocalLSR} and \ref{AdvancedLifting} there is \(h_{s_0} \in T_{d + 1, -n}(R_{d + 1, -n})\) such that \(e_{(-n)'} e_{[1d]'} h_{s_0} g_{t - 1} \cdots g_1 x\) is unimodular in \(R \otimes \kappa(\mathfrak q_t)\) and \(\beta(h_{s_0}) \in R \mathfrak q_1 \cdots \mathfrak q_{t - 1} \mathfrak p_1 \cdots \mathfrak p_{s_0}\). But now \(e_{1 n} h_{s_0} g_{t - 1} \cdots g_1 x\) is not necessarily invertible in \(R_1 \otimes \kappa(\mathfrak p_s)\) for \(s > s_0\) (though it is for \(s \leq s_0\)). By lemmas \ref{SemilocalLSR} and \ref{AdvancedLifting} there are \(h_{s_0 + 1}, \ldots, h_N \in T_{d + 1, -n}(R_{d + 1, -n})\) such that \(\beta(h_s) \in R\mathfrak q_1 \cdots \mathfrak q_{M'} \mathfrak p_1 \cdots \mathfrak p_{s - 1}\) and \(e_{1 n} h_s \cdots h_{s_0} g_{t - 1} \cdots g_1 x\) are invertible in \(R_1 \otimes \kappa(\mathfrak p_s)\) for all \(s > s_0\) (note that the elements \(e_{1 n} (h_{s - 1} \cdots h_{s_0})^{-1} h_{s - 1} \cdots h_{s_0} g_{t - 1} \cdots g_1 x\) are invertible in these algebras, hence we apply lemma \ref{AdvancedLifting} in order to lift \((h_{s - 1} \cdots h_{s_0})^{-1}\)). Finally, let \(g_t = h_N \cdots h_{s_0}\).

Without loss of generality we may assume that \(e_{(-n)'} e_{[1d]'} x\) is unimodular in all \(R \otimes \kappa(\mathfrak q_t)\). By lemma \ref{UnimodularityLocus} there is an ideal \(\mathfrak a \leq \mathfrak b \leq K\) such that \(\mathfrak b \not \leq \mathfrak q_t\) for all \(t\) and \(e_{(-n)'} e_{[1d]'} x\) is unimodular in \(R \otimes K_b\) for all \(b \in \mathfrak b\). Note that \(\mathrm{BS}(K / \mathfrak b) < d\) since the space \(\mathrm{Max}(K / \mathfrak b)\) is a closed subspace of \(\mathrm{Max}(K / \mathfrak a)\) not containing any \(X_t\). By the induction assumption, there is \(g \in \langle T_{i, -n}(R_{i, -n}), T_{i j}(R_{i j}) \mid 1 \leq i, j \leq d; i \neq j \rangle\) such that \(e_{1 n} gx\) is invertible in \(R_1 \otimes \kappa(\mathfrak p_s)\) for all \(s\) and \(e_{(-n)'} gx\) is unimodular in \(R \otimes K / \mathfrak b\). The same lemma implies that \(e_{(-n)'} gx\) is unimodular in \(R \otimes K / \mathfrak a\).
\end{proof}

Now let us prove the main result of this section.

\begin{theorem}\label{DimensionLSR}
Let \((R, \Delta)\) be a quasi-finite odd form \(K\)-algebra with a free orthogonal hyperbolic family of rank \(n \geq 1\). Suppose that \(\mathrm{BS}(K) \leq d\). Then \(\Lambda\mathrm{sr}(\eta_1; R, \Delta) \leq d + 1\).
\end{theorem}
\begin{proof}
By lemma \ref{LimitLSR} we may assume that \((R, \Delta)\) is finite (with the same orthogonal hyperbolic family). The proof is by induction on \(d\). The case \(d = -1\) trivial and the case \(d = 0\) is lemma \ref{SemilocalLSR}, hence we may assume that \(d \geq 0\) and \(n = d + 2\). Of course, we already assume that \(\mathrm{sr}(R_1) \leq d + 1\), though this can be proved in the same way. Let \(x \in R_{0' 1}\) be unimodular and \(\mathrm{Max}(K) = \bigcup_{s = 1}^N X_s\) be the decomposition into a finite union of irredicuble Noetherian subspaces of dimensions at most \(d\). The set \(\{\bigcap_{\mathfrak m \in X_s} \mathfrak m \mid 1 \leq s \leq N\}\) consists of prime ideals of \(K\). Let \(\mathfrak p_1, \ldots, \mathfrak p_{N'}\) be all elements of this set without repetitions is such an order that \(\mathfrak p_{s_1} \not \leq \mathfrak p_{s_2}\) for \(s_1 < s_2\) (it is possible that \(N' < N\)). By lemmas \ref{SemilocalLSR} and \ref{AdvancedLifting} there are \(g_1, \ldots, g_{N'} \in \unit^-(R^{\mathrm{ev}}, \Delta^{\mathrm{ev}})\) such that \(e_{1 n} g_s \cdots g_1 x\) are invertible in \(R_1 \otimes \kappa(\mathfrak p_s)\) and \(\beta(g_s) \in R \mathfrak p_1 \cdots \mathfrak p_{s - 1}\) for all \(s\). Hence without loss of generality \(e_{1 n} x\) is invertible in all \(R \otimes \kappa(\mathfrak p_s)\).

Now by lemmas \ref{SemilocalLSR} and \ref{AdvancedLifting} there are \(g_1, \ldots, g_{N'} \in T_{n - 1, n}(R_{n - 1, n})\) such that \(e_{1, n - 1}g_s \cdots g_1 x\) is invertible in \(R \otimes \kappa(\mathfrak p_s)\) and \(\beta(g_s) \in R \mathfrak p_1 \cdots \mathfrak p_{s - 1}\). Hence we may assume that \(e_{1, n - 1} x\) is also invertible in all \(R_1 \otimes \kappa(\mathfrak p_s)\). By lemma \ref{UnimodularityLocus} there is an ideal \(\mathfrak a \leq K\) such that \(e_{1, n - 1} x\) is invertible in \(R \otimes K_a\) for all \(a \in \mathfrak a\) and \(\mathfrak a \not \leq \mathfrak p_s\) for all \(s\). Note that \(\mathrm{BS}(K / \mathfrak a) < d\). Applying lemma \ref{DimensionSR}, we obtain \(g \in \langle T_{i, -n}(R_{i, -n}), T_{i j}(R_{i j}) \mid 1 \leq i, j \leq d; i \neq j \rangle\) such that \(e_{1 n} g x\) is invertible in all \(R_1 \otimes \kappa(\mathfrak p_s)\) and \(e_{(-n)'} g x\) is unimodular in \(R \otimes K / \mathfrak a\). By lemma \ref{UnimodularityLocus} \(e_{(-n)'} g x\) is unimodular in \(R\).

Without loss of generality, we may assume that \(e_{1 n} x\) is invertible in all \(R_1 \otimes \kappa(\mathfrak p_s)\) and \(e_{(-n)'} x\) is unimodular in \(R\). By lemma \ref{UnimodularityLocus} there is \(\mathfrak a \leqt K\) such that \(\mathfrak a \not \leq \mathfrak p_s\) for all \(s\) and \(e_{1 n} x\) is invertible in \(R_1 \otimes K_a\) for all \(a \in \mathfrak a\). Since \(\mathrm{BS}(K / \mathfrak a) < d\), by the induction assumption there is \(g \in \langle T_{i j}(R_{i j}), T_j(\Delta^{\mathrm{ev}}_j) \mid 1 \leq i, -j \leq d + 1; i \neq j \rangle\) such that \(e_+ gx\) is unimodular in \(R \otimes K / \mathfrak a\). But then \(e_+ gx\) is unimodular in \(R\) by the same lemma.
\end{proof}

\section{Stability for \(\mathrm K_1\)-group}

In this section we will prove stability for the factor \(\unit(n; R, \Delta) / \eunit(n; R, \Delta)\) under the  increasing of \(n\). By default, \((R, \Delta)\) will mean an odd form ring with an orthogonal hyperbolic family of rank \(n \geq 1\).

Now let \(l \neq 0\) be an index (usually \(l = \pm n\)). The Heisenberg subgroup is
\begin{align*}
H_l(n; R, \Delta) &= \{g \in \unit(n; R, \Delta) \mid \beta_{(-l)', l'}(g) = \beta_{-l}(g) = \beta_l(g) = 0; \gamma^\circ_{l'}(g) = \dot 0\} = \\
&= \{g \in \unit(n; R, \Delta) \mid \beta_{(-l)', l'}(g) = \beta_{-l}(g) = \beta_l(g) = 0; \gamma^0_{l'}(g) = \dot 0\} = \\
&= \{g \in \unit(n; R, \Delta) \mid\\
&\qquad \beta_{(-l)', l'}(g) = \beta_{-l}(g) = \beta_l(g) = 0; \gamma_{l'}(g) = q_{-l} \cdot \beta(g) e_{l'}\}.
\end{align*}

Note that if \((R, \Delta)\) is the split symplectic algebra over \(K\) from example \ref{SymplecticOFA}, then \(R_l \cong K\) has trivial \((-1)\)-involution and \(H_l(n; R, \Delta)\) is isomorphic to the ordinary Heisenberg group \(\Heis(R_{|l|', l}, B)\) for a right hermitian module over \(K\) (the symplectic form is \(B(x, y) = e_{l, -l} \inv x y\)).

It is easy to see that \(H_{\pm n}(n; R, \Delta)\) for \(n = 1\) are actually the two groups of elementary transvections. By lemmas \ref{TransvectionRelations} and \ref{HyperbolicGluing}, \(H_n(n; R, \Delta)\) is normalized by \(\unit(n - 1; R, \Delta)\) and is generated by transvections \(T_{i n}(x)\), \(T_n(u)\) with the obvious relations, and similarly for \(H_{-n}(n; R, \Delta)\). It follows from lemma \ref{TransvectionRelations} that \(\eunit(n; R, \Delta) = \langle H_n(n; R, \Delta), H_{-n}(n; R, \Delta) \rangle\) for all \(n \geq 1\). In particular, \(\eunit(n; R, \Delta)\) is normalized by \(\unit(n - 1; R, \Delta)\).

As usual, \(\kunit(n; R, \Delta) = \unit(n; R, \Delta) / \eunit(n; R, \Delta)\) for an odd form ring \((R, \Delta)\) and \(\kunit(n; R, \Delta; I, \Gamma) = \unit(n; I, \Gamma) / \eunit(n; R, \Delta; I, \Gamma)\) for an odd form ideal \((I, \Gamma) \leqt (R, \Delta)\). In general these objects are just pointed sets. Note that there is a sequence
\[1 \rar \kunit(n; R, \Delta; I, \Gamma) \xrightarrow{(p_1^{-1})_*} \kunit(n; I \rtimes R, \Gamma \rtimes \Delta) \xrightarrow{(p_2)_*} \kunit(n; R, \Delta) \rar 1,\]
see lemma \ref{Relativization}. It follows from a simple diagram chasing that this sequence is short exact in the following sense: \((p_1^{-1})_*\) is injective, \((p_2)_*\) is surjective, the image of \((p_1^{-1})_*\) is exactly the preimage of the distinguished point under \((p_2)_*\). Moreover, there is a \(\unit(n; R, \Delta)\)-equivariant section \(d_* \colon \kunit(n; R, \Delta) \rar \kunit(n; I \rtimes R, \Gamma \rtimes \Delta)\) of \((p_2)_*\) and the action of \(\unit(n; R, \Delta)\) on \(\kunit(n; I \rtimes R, \Gamma \rtimes \Delta)\) induce a transitive action on the fibers of \((p_2)_*\), hence there is a bijection \(\kunit(n; I \rtimes R, \Gamma \rtimes \Delta) \cong \kunit(n; R, \Delta; I, \Gamma) \times \kunit(n; R, \Delta)\). The bijection is not canonical in general, though if all elementary groups are normal in the corresponding unitary groups, then the sequence is split short exact with the splitting \(d_*\).

Since \(\unit(n - 1; R, \Delta)\) normalizes \(\eunit(n; R, \Delta)\), the next theorem shows that \(\eunit(n; R, \Delta; I, \Gamma) \leqt \unit(n; R, \Delta)\) under an assumption on the \(\Lambda\)-stable rank (in the relative case by the short exact sequence for elementary groups). Of course, the normality holds under much weaker assumptions.

\begin{theorem}\label{SurjectiveKU}
Suppose that \((R, \Delta)\) is an odd form ring with a free orthogonal hyperbolic family of rank \(n\) and that \(\Lambda\mathrm{sr}(\eta_1; R, \Delta) \leq n - 1\). Then the natural map \(\kunit(n - 1; R, \Delta) \rar \kunit(n; R, \Delta)\) is surjective. Also, if \((I, \Gamma) \leqt (R, \Delta)\) is an odd form ideal and \(\Lambda\mathrm{sr}(\eta_1; I \rtimes R, \Gamma \rtimes \Delta) \leq n - 1\) (with no restriction on \(\Lambda\mathrm{sr}(\eta_1; R, \Delta)\)), then the natural map \(\kunit(n - 1; R, \Delta; I, \Gamma) \rar \kunit(n; R, \Delta; I, \Gamma)\) is surjective.
\end{theorem}
\begin{proof}
We have to prove that every \(g \in \unit(n; R, \Delta)\) actually lies in \(\unit(n - 1; R, \Delta) \eunit(n; R, \Delta)\). First of all, suppose that \(\alpha_n(g) = e_n\). In this case there is unique \(h \in H_n(n; R, \Delta)\) such that \(\beta_{* n}(g') = 0\) and \(\gamma^\circ_n(g') = \dot 0\) for \(g' = h^{-1} g\). Indeed, we have to take \(\beta_{* n}(h) = \beta_{* n}(g)\) and \(\gamma^\circ_n(h) = \gamma^\circ_n(g)\), these conditions determine unique \(h \in H_n(n; R, \Delta)\). But then \((g')^{-1}\) satisfies the same equations (the equations say essentially that \(g'\) is block-triangular matrix and one of the diagonal blocks is trivial), hence \(\beta_{-n, *}(g') = 0\). Similarly, there is \(h' \in H_{-n}(n; R, \Delta)\) such that \(\beta_{n *}(g'') = 0\) and \(\gamma^\circ_{-n}(g'') = \dot 0\) for \(g'' = g' h'\) (i.e. \(\beta_{*, -n}(h') = \beta_{*, -n}((g')^{-1})\) and \(\gamma^\circ_{-n}(h') = \gamma^\circ_{-n}((g')^{-1})\)). This means that \(g'' \in \unit(n - 1; R, \Delta)\).

In the general case we will multiply \(g\) from the left by elementary transvections until \(\alpha_n(g)\) will be equal to \(e_n\). Note that the sequence \(\{\alpha_{n i}(g^{-1}) \alpha_{i n}(g)\}_{i = -n}^0 \sqcup \{\alpha_{i n}(g)\}_{i = 1}^n\) is unimodular. By \(\mathrm{sr}(R_1) \leq n - 1\) and lemma \ref{IncreasingSR} there is \(y \in R_{+, n}\) such that \(\{\alpha_{i n}(g)\}_{i = -n}^{-1} \sqcup \{\alpha_{i n}(g) + e_i y \alpha_{n 0}(g^{-1}) \alpha_{0 n}(g)\}_{i = 1}^n\) is unimodular. Hence there is \(h \in \eunit(n; R, \Delta)\) such that \(\alpha_{0' n}(g')\) is unimodular for \(g' = hg\). Now recall that \(\Lambda \mathrm{sr}(\eta_1; R, \Delta) \leq n - 1\), so there is \(h' \in \eunit(n; R, \Delta)\) such that \(\alpha_{+, n}(g'')\) is unimodular for \(g'' = h'g'\). The condition \(\mathrm{sr}(R_1) \leq n - 1\) also implies that there is \(h'' \in \eunit(n; R, \Delta)\) such that \(e_{n'} \alpha_{+, n}(g''')\) is unimodular for \(g''' = h'' g''\), and then clearly exists \(h''' \in \eunit(n; R, \Delta)\) such that \(\alpha_n(h''' g''') = e_n\).

The relative variant follows from the diagram chasing using the short exact sequences.
\end{proof}

In order to prove the injective stability we show that the elementary group decomposes as a product of certain subgroups. Let
\begin{align*}
A_1 &= \eunit(n - 1; R, \Delta) H_n(n; R, \Delta),\\
A_2 &= \langle T_{i, -n}(R_{i, -n}), T_{-n}(\Delta^0_{-n}) \mid i \geq -1 \rangle \leq H_{-n}(n; R, \Delta),\\
A_3 &= \langle T_{i j}(R_{i j}) \mid i, j \geq 2 \rangle,\\
A_4 &= \langle T_{i j}(R_{i j}), T_j(\Delta^0_j) \mid i \leq 1 \text{ and } j \geq 3 \rangle,\\
A_5 &= \langle T_{i j}(R_{i j}), T_j(\Delta^0_j) \mid j \geq 2 \rangle.
\end{align*}
Informally, \(A_1\), \(A_2\), \(A_3\), \(A_4\), and \(A_5\) are the groups of elementary matrices with the following \(\alpha\)-s for \(n = 3\) (we use the numeration \(-n, \ldots, n\) for rows and columns):
\[
\left(\begin{smallmatrix}
1 & * & * & * & * & * & * \\
0 & * & * & * & * & * & * \\
0 & * & * & * & * & * & * \\
0 & * & * & * & * & * & * \\
0 & * & * & * & * & * & * \\
0 & * & * & * & * & * & * \\
0 & 0 & 0 & 0 & 0 & 0 & 1
\end{smallmatrix}\right),
\left(\begin{smallmatrix}
1 & 0 & 0 & 0 & 0 & 0 & 0 \\
0 & 1 & 0 & 0 & 0 & 0 & 0 \\
* & 0 & 1 & 0 & 0 & 0 & 0 \\
* & 0 & 0 & 1 & 0 & 0 & 0 \\
* & 0 & 0 & 0 & 1 & 0 & 0 \\
* & 0 & 0 & 0 & 0 & 1 & 0 \\
* & * & * & * & * & 0 & 1
\end{smallmatrix}\right),
\left(\begin{smallmatrix}
* & * & 0 & 0 & 0 & 0 & 0 \\
* & * & 0 & 0 & 0 & 0 & 0 \\
0 & 0 & 1 & 0 & 0 & 0 & 0 \\
0 & 0 & 0 & 1 & 0 & 0 & 0 \\
0 & 0 & 0 & 0 & 1 & 0 & 0 \\
0 & 0 & 0 & 0 & 0 & * & * \\
0 & 0 & 0 & 0 & 0 & * & *
\end{smallmatrix}\right),
\left(\begin{smallmatrix}
1 & 0 & * & * & * & * & * \\
0 & 1 & 0 & 0 & 0 & 0 & * \\
0 & 0 & 1 & 0 & 0 & 0 & * \\
0 & 0 & 0 & 1 & 0 & 0 & * \\
0 & 0 & 0 & 0 & 1 & 0 & * \\
0 & 0 & 0 & 0 & 0 & 1 & 0 \\
0 & 0 & 0 & 0 & 0 & 0 & 1
\end{smallmatrix}\right),
\left(\begin{smallmatrix}
* & * & * & * & * & * & * \\
* & * & * & * & * & * & * \\
0 & 0 & 1 & 0 & 0 & * & * \\
0 & 0 & 0 & 1 & 0 & * & * \\
0 & 0 & 0 & 0 & 1 & * & * \\
0 & 0 & 0 & 0 & 0 & * & * \\
0 & 0 & 0 & 0 & 0 & * & *
\end{smallmatrix}\right).
\]

We are going to prove that \(\eunit(n; R, \Delta) = A_1 A_4 A_5\) under a suitable assumption on the stable rank. In the next two lemmas let \(T'_-(u)\), \(T'_+(u)\), and \(D'_+(a)\) be the elementary transvections and dilations with respect to the orthogonal hyperbolic family \(\eta_2 + \ldots + \eta_n\) of rank \(1\), see lemma \ref{HyperbolicGluing} for their properties. Let also \(T'_-(*)\), \(T'_+(*)\), and \(D'_+(*) = A_3\) be the groups of these elements. Note that \(A_5 = T'_+(*) \rtimes D'_+(*)\) and \(D'_-(*) \subseteq A_1 A_2\).

\begin{lemma}\label{ReducedA}
Let \((R, \Delta)\) be an odd form ring with a free orthogonal hyperbolic family of rank \(n\) and suppose that \(\mathrm{sr}(R_1) \leq n - 2\). Then \(A_1 A_2 A_5 \subseteq A_1 A_2 A_3 A_4\).
\end{lemma}
\begin{proof}
Suppose that \(g = a_1 a_2 a_5\), where \(a_i \in A_i\). Clearly, \(a_5 = D'_+(b) T'_+(u)\) for uniquely determined \(u\) and \(b\). By the stable rank condition there is \(h \in H = \langle T_{2 n}(R_{2 n}), \ldots, T_{n - 1, n}(R_{n - 1, n}) \rangle\) such that \(h D'_+(b) = D'_+(b')\) with the sequence \(b'_{22}, \ldots, b'_{n - 1, 2}\) being unimodular. Note that \(a_1 h^{-1} \in A_1\) and \(\up h {a_2} \in T'_-(*)\), hence \(g = (a_1 h^{-1})\, \up h {a_2} D'_+(b') T'_+(u) = a'_1 a'_2 D'_+(b') T'_+(u)\) for some \(a'_i \in A_i\).

Since \(b'_{22}, \ldots, b'_{n - 1, 2}\) is unimodular, there is \(c \in e_2 R (e_2 + \ldots + e_{n - 1})\) such that \(c b' e_2 = e_2\). Let \(v = u \cdot c\), then \(v \cdot e_n = \dot 0\) and
\[g = a'_1 a'_2 D'_+(b') T'_+(u) = a'_1 T'_+(v)\, \up{T'_+(\dotminus v)}{a'_2} D'_+(b') T'_+(u \dotminus u \cdot cb').\]
Note that \(a'_1 T'_+(v) \in A_1\), \(\up{T'_+(\dotminus v)}{a'_2} \in H_{-n}(n; R, \Delta) \subseteq A_2 A_3\), \(D'_+(b') \in A_3\), and \(T'_+(u \dotminus u \cdot cb') \in A_4\). In other words, \(g \in A_1 A_2 A_3 A_4\).
\end{proof}

\begin{lemma}\label{DecompositionA}
Let \((R, \Delta)\) be an odd form ring with a free orthogonal hyperbolic family of rank \(n\) and suppose that \(\mathrm{sr}(R_1) \leq n - 2\). Then \(\eunit(n; R, \Delta) = A_1 A_2 A_5\).
\end{lemma}
\begin{proof}
Note that \(\eunit(n; R, \Delta)\) is generated by \(A_5\) and \(\widetilde T = \langle T_{-2}(\Delta^0_{-2}), T_{-1, -2}(R_{-1, -2}), T_{1, -2}(R_{1, -2}) \rangle\). Clearly, \(1 \in A_1 A_2 A_5\) and the set \(A_1 A_2 A_5\) is closed under multiplications by elements of \(A_5\) from the right. By lemma \ref{ReducedA} it remains to prove that \(A_1 A_2 A_3 A_4 t \subseteq A_1 A_2 A_5\) for any \(t \in \widetilde T\). Note that \(\up {t^{-1}}{A_4} \leq A_5\) and \(A_2 A_3 t \leq T'_-(*) A_3\). Hence \(A_1 A_2 A_3 A_4 t = A_1 (A_2 A_3 t) (\up {t^{-1}}{A_2}) \subseteq A_1 (A_1 A_2 A_3) A_5 = A_1 A_2 A_5\).
\end{proof}

Finally, let us prove the injective stability.

\begin{theorem}\label{InjectiveKU}
Let \((R, \Delta)\) be an odd form ring with a free orthogonal family of rank \(n\) and suppose that \(\mathrm{sr}(R_1) \leq n - 2\). Then the map \(\kunit(n - 1; R, \Delta) \rar \kunit(n; R, \Delta)\) is injective. Also, if \((I, \Gamma) \leqt (R, \Delta)\) and \(\mathrm{sr}((I \rtimes R)_1) \leq n - 2\) (with no restriction on \(\mathrm{sr}(R_1)\)), then the map \(\kunit(n - 1; R, \Delta; I, \Gamma) \rar \kunit(n; R, \Delta; I, \Gamma)\) is also injective.
\end{theorem}
\begin{proof}
In the absolute case we need to prove that every \(g \in \eunit(n; R, \Delta) \cap \unit(n - 1; R, \Delta)\) lies in \(\eunit(n - 1; R, \Delta)\). Lemma \ref{DecompositionA} implies that \(g = a_1 a_2 a_5\) for \(a_i \in A_i\). But then \(a_1 a_2 a_5 \in \unit(n - 1; R, \Delta)\) implies that \(a_2 = 1\), hence we may assume that \(a_1 \in \eunit(n - 1; R, \Delta)\) (the factor from \(H_n(n; R, \Delta)\) may be pushed into \(a_5\)) and therefore \(a_5 \in \unit(n - 1; R, \Delta)\). We have to prove that \(a_5 \in \eunit(n - 1; R, \Delta)\).

Multiplying by elementary transvections from \(\eunit(n - 1; R, \Delta)\) we may even assume that \(a_5 \in \unit(n - 1; R, \Delta) \cap \langle T_{i j}(R_{i j}) \mid 2 \leq i, j \leq n \rangle\). Write \(a_5 = \prod_s t_s\), where \(t_s = T_{i_s j_s}(x_s)\) for \(2 \leq i_s, j_s \leq n\). Now let \(t'_s = t_s\) if \(i_s, j_s \neq n\), \(t'_s = T_{i_s 1}(x_s e_{n 1})\) if \(j_s = n\), and \(t'_s = T_{1 j_s}(e_{1 n} x_s)\) if \(i_s = n\). It is easy to see that \(\prod_s t'_s = a_1\), i.e. \(a_5 \in \eunit(n - 1; R, \Delta)\).

The relative case follows from a simple diagram chasing.
\end{proof}

\section{Stability for form morphisms}

In this section \((R, \Delta)\) and \((S, \Theta)\) are odd form rings with free orthogonal families of ranks \(n \geq 1\) and \(m \geq 1\). Let \(\Hom(m; S, \Theta; n; R, \Delta) = \Hom(S, \Theta; R, \Delta)\) and \(\Hom(m - 1; S, \Theta; n - 1; R, \Delta) = \Hom(S_{|m|'}, \Theta_{|m|'}^{|m|'}; R_{|n|'}, \Delta_{|n|'}^{|n|'})\). There is a natural left action of \(\unit(n; R, \Delta)\) on \(\Hom(m; S, \Theta; n; R, \Delta)\), let \(\unit(n; R, \Delta) \backslash {\Hom(m; S, \Theta; n; R, \Delta)}\) be the set of orbits under this action. Clearly, if two morphisms lie in the same orbit under the right action of \(\unit(m; S, \Theta)\), they also lie in the same orbit under the left action of \(\unit(n; R, \Delta)\).

Stability for \(\unit(n; R, \Delta) \backslash {\Hom(m; S, \Theta; n; R, \Delta)}\) or \(\unit(n; R, \Delta) \backslash {\punit(n; R, \Delta)}\) does not holds in general even for one-dimensional base commutative rings, see example \ref{Twisting}. Instead of \(\Hom(m; S, \Theta; n; R, \Delta)\) we consider the set
\begin{align*}
\Hom^*(m; S, \Theta; n; R, \Delta) = \{&f \in \Hom(m; S, \Theta; n; R, \Delta) \mid \unit(\up f{e_{|m|}}, e_{|n|}; R, \Delta) \neq \varnothing\}.
\end{align*}
It is easy to see that this set is closed under the action of \(\unit(n; R, \Delta)\). Also, this set is independent on the order of hyperbolic pairs. We will prove the stabilization for \(\unit(n; R, \Delta) \backslash {\Hom^*(m; S, \Theta; n; R, \Delta)}\).

Let also
\[\Hom_{i j}(m; S, \Theta; n; R, \Delta) = \{f \in \Hom(m; S, \Theta; n; R, \Delta) \mid \up f {e_{|i|}} = e_{|j|}\}\]
for any \(1 \leq i \leq m\) and \(1 \leq j \leq n\), this set is closed under the action of \(\unit(n - 1; R, \Delta)\). Obviously, \(\Hom_{m n}(m; S, \Theta; n; R, \Delta) \subseteq \Hom^*(m; S, \Theta; n; R, \Delta)\). Moreover, there is a canonical map \(\Hom_{m n}(m; S, \Theta; n; R, \Delta) \rar \Hom^*(m - 1; S, \Theta; n - 1; R, \Delta)\) if \(m, n \geq 2\) (and into the non-starred right hand side for all \(m, n\)).

Similarly, we may define corresponding projective unitary groups. Let
\begin{align*}
\punit^*(n; R, \Delta) &= \Hom^*(n; R, \Delta; n; R, \Delta) \cap \punit(n; R, \Delta),\\
\punit_n(n; R, \Delta) &= \Hom_{nn}(n; R, \Delta; n; R, \Delta) \cap \punit(n; R, \Delta),\\
\punit^*(n; R, \Delta; I, \Gamma) &= \punit^*(n; I \rtimes R, \Gamma \rtimes \Delta) \cap \punit(n; R, \Delta; I, \Gamma),\\
\punit_n(n; R, \Delta; I, \Gamma) &= \punit_n(n; I \rtimes R, \Gamma \rtimes \Delta) \cap \punit(n; R, \Delta; I, \Gamma).
\end{align*}
All these objects are actually groups. Finally, the groups \(\gunit^*(n; T, \Xi; R, \Delta)\), \(\gunit_n(n; T, \Xi; R, \Delta)\), \(\gunit^*(n; T, \Xi; R, \Delta; I, \Gamma)\), and \(\gunit_n(n; T, \Xi; R, \Delta; I, \Gamma)\) are the preimages of the corresponding projective unitary groups under the canonical map \(\gunit(n; T, \Xi; R, \Delta) \rar \punit(n; R, \Delta)\). Note that 
\[\gunit_n(n; T, \Xi; R, \Delta) = \{g \in \gunit(n; T, \Xi; R, \Delta) \mid \beta_{|n|, |n|'}(g) = \beta_{|n|', |n|}(g) = 0\}.\]

Let us now show that in certain situations the stabilization still can be proved for \({\unit(n; R, \Delta)} \backslash {\Hom(m; S, \Theta; n; R, \Delta)}\). Recall that the Picard group \(\mathrm{Pic}_K(R)\) is the group of isomorphism classes of invertible \(R\)-\(R\)-bimodules that are also \(K\)-modules, where \(R\) is arbitrary unital \(K\)-algebra. Also there is an exact sequence \(R^* \rar \Aut_K(R) \rar \mathrm{Pic}_K(R)\), the right map is given by \(\alpha \mapsto R_\alpha\), where \(R_\alpha = \{r_\alpha \mid r \in R\}\) has the obvious \(K\)-module structure and the bimodule structure \(xm_\alpha y = (xm\, \up \alpha y)_\alpha\). See \cite{Bass}, chapter II for details. We denote the outer automorphism group of \(R\) by \(\mathrm{Out}_K(R)\), it is isomorphic to a subgroup of \(\mathrm{Pic}_K(R)\). Now let \((R, \Delta)\) and \((S, \Theta)\) be odd form \(K\)-algebras with free orthogonal hyperbolic families of ranks \(n = m \geq 1\). Suppose that the family of \((R, \Delta)\) is Morita complete, \(R_1 \cong R_{-1}\) as \(K\)-algebra, and for every \(\sigma \in \Hom(S, \Theta; R, \Delta)\) (or \(\sigma \in \punit(R, \Delta)\)) either \(\up \sigma {e_1}\) is Morita equivalent to \(e_1\) or \(\up \sigma {e_{-1}}\) is Morita equivalent to \(e_1\) (this holds for all classical odd form algebras over a commutative ring with connected spectrum, see the examples). In general \(e_1 R\) may not be isomorphic to \(\up \sigma {e_1} R\) or \(\up \sigma {e_{-1}} R\) as an \(R\)-module, but under our assumption the only obstacle is described in terms of the class \([e_1 R\, \up \sigma {e_{\pm 1}}] \in \mathrm{Pic}_K(R_1) / \mathrm{Out}_K(R_1)\) (note that \(\Hom_R(e_1 R, \up \sigma {e_{\pm 1}} R) \cong \up \sigma {e_{\pm 1}} R e_1\) canonically). If \(K\) is local or a PID and \(R_1 \cong K\), then \(\mathrm{Pic}_K(R_1)\) is trivial, hence \(\up \sigma {\eta_1}\) and \(\eta_1\) are isomorphic for all \(\sigma\), i.e. \(\Hom^*(S, \Theta; R, \Delta) = \Hom(S, \Theta; R, \Delta)\) (or \(\punit^*(n; R, \Delta) = \punit(R, \Delta)\)) in this case.

Let us begin the proof of the stabilization with a preliminary lemma.

\begin{lemma}\label{UnitaryExtension}
Let \((R, \Delta)\) be an odd form ring with a free orthogonal family of rank \(n \geq 1\). Suppose that \(\Lambda\mathrm{sr}(\eta_1; R, \Delta) \leq n - 1\). Then for every hermitian idempotent \(\eps \in R \rtimes K\) and for every \(g \in \unit(\eps, e_{|n|}; R, \Delta)\) there is \(\widetilde g \in \eunit(R, \Delta)\) such that \(\alpha(g) = \alpha(\widetilde g) e_{|n|}\), \(\gamma(g) = \gamma(\widetilde g) \cdot e_{|n|}\), and \(\eps = \up{\widetilde g}{e_{|n|}}\).
\end{lemma}
\begin{proof}
It suffices to find \(h \in \eunit(R, \Delta)\) such that \(\alpha(h) \alpha(g) = e_{|n|}\) and \(\gamma(h) \cdot \alpha(g) \dotplus \gamma(g) = \dot 0\), because then we may take \(\widetilde g = h^{-1}\). Note that \(\alpha(g) e_n\) is unimodular (more precisely, \((e_n \alpha(g^{-1})) (\alpha(g) e_n) = e_n\)), \(e_n \alpha(g^{-1}) e_0 \in \inv{\pi(\Delta^0_n)}\), and \(\Lambda\mathrm{sr}(\eta_1; R, \Delta) \leq n - 1\), hence there is \(h_1 \in \eunit(n; R, \Delta)\) such that \(e_n \alpha(h_1) \alpha(g) e_n = e_n\) (see the proof of theorem \ref{SurjectiveKU}). There is \(h_2 \in H_n(n; R, \Delta)\) such that \(\alpha_{*n}(h_2^{-1}) = \alpha(h_1) \alpha(g) e_n\) and \(\gamma_n(h_2^{-1}) = \gamma(h_1) \cdot \alpha(g) e_n \dotplus \gamma(g) \cdot e_n\). It follows that \(\alpha(h_2 h_1) \alpha(g) e_n = e_n\) and \(\gamma(h_2 h_1) \cdot \alpha(g) e_n \dotplus \gamma(g) \cdot e_n = \dot 0\).

But now obviously \(e_{-n} \alpha(h_2h_1) \alpha(g) = e_{-n}\). Similarly, there is \(h_3 \in H_{-n}(n; R, \Delta)\) such that \(\alpha_{*, -n}(h_3^{-1}) = \alpha(h_2 h_1) \alpha(g) e_{-n}\) and \(\gamma_{-n}(h_3^{-1}) = \gamma(h_2 h_1) \cdot \alpha(g) e_{-n} \dotplus \gamma(g) \cdot e_{-n}\). Hence we may take \(h = h_3 h_2 h_1\).
\end{proof}

\begin{prop}\label{HomDecomposition}
Let \((S, \Theta)\) and \((R, \Delta)\) be odd form rings with free orthogonal hyperbolic families of ranks \(m \geq 1\) and \(n \geq 1\). Suppose that \(\Lambda\mathrm{sr}(\eta_1; R, \Delta) \leq n - 1\). Then
\[\Hom^*(m; S, \Theta; n; R, \Delta) = \eunit(n; R, \Delta) \Hom_{m n}(m; S, \Theta; n; R, \Delta).\]
\end{prop}
\begin{proof}
Let \(f \in \Hom^*(m; S, \Theta; n; R, \Delta)\) and \(g \in \unit(\up f{e_{|m|}}, e_{|n|}; R, \Delta)\). By lemma \ref{UnitaryExtension} there is \(\widetilde g \in \eunit(n; R, \Delta)\) such that \(\up f {e_{|m|}} = \up{\widetilde g}{e_{|n|}}\). In other words, \(\widetilde g^{-1} f \in \Hom_{mn}(m; S, \Theta; n; R, \Delta)\).
\end{proof}

\begin{theorem}\label{HomStability}
Let \((S, \Theta)\) and \((R, \Delta)\) be odd form rings with free orthogonal hyperbolic families of ranks \(m \geq 1\) and \(n \geq 1\), \((R, \Delta) \subseteq (T, \Xi)\) be an odd form overring, and \((I, \Gamma) \leqt (R, \Delta)\). Let also
\begin{align*}
\unit'_{|n|} &= \unit(R_{|n|}, \Delta_{|n|}^{|n|}) \times \unit(n - 1; R, \Delta) = \gunit_n(n; R, \Delta; R, \Delta),\\
\unit'_{|n|}(I, \Gamma) &= \unit(I_{|n|}, \Gamma_{|n|}^{|n|}) \times \unit(n - 1; I, \Gamma).
\end{align*}
Then the following holds:
\begin{enumerate}
\item If \(\Lambda\mathrm{sr}(\eta_1; R, \Delta) \leq n - 1\), then the canonical maps
\begin{align*}
{\unit'_{|n|}} \backslash {\Hom_{m n}(m; S, \Theta; n; R, \Delta)} &\rar \unit(n; R, \Delta) \backslash {\Hom^*(m; S, \Theta; n; R, \Delta)},\\
\punit_n(n; R, \Delta) / \unit'_{|n|} &\rar \punit^*(n; R, \Delta) / \unit(n; R, \Delta),\\
\gunit_n(n; T, \Xi; R, \Delta) / \unit'_{|n|} &\rar \gunit^*(n; T, \Xi; R, \Delta) / \unit(n; R, \Delta)
\end{align*}
are bijective (they are always injective).
\item If \(m, n \geq 2\) and \(\Lambda\mathrm{sr}(\eta_1; R, \Delta) \leq n - 2\), then the canonical maps
\begin{align*}
{\unit'_{|n|}} \backslash {\Hom_{m n}(m; S, \Theta; n; R, \Delta)} &\rar \unit(n - 1; R, \Delta) \backslash {\Hom^*(m - 1; S, \Theta; n - 1; R, \Delta)},\\
\punit_n(n; R, \Delta) / \unit'_{|n|} &\rar \punit^*(n - 1; R, \Delta) / \unit(n - 1; R, \Delta),\\
\gunit_n(n; T, \Xi; R, \Delta) / \unit'_{|n|} &\rar \gunit^*(n - 1; T, \Xi; R, \Delta) / \unit(n - 1; R, \Delta)
\end{align*}
are also bijective (they are injective for \(m, n \geq 2\) without restriction on the \(\Lambda\)-stable rank).
\item If \(\Lambda\mathrm{sr}(\eta_1; I \rtimes R, \Gamma \rtimes \Delta) \leq n - 1\), then the relative maps
\begin{align*}
\punit_n(n; R, \Delta; I, \Gamma) / \unit'_{|n|}(I, \Gamma) &\rar \punit^*(n; R, \Delta; I, \Gamma) / \unit(n; I, \Gamma),\\
\gunit_n(n; T, \Xi; R, \Delta; I, \Gamma) / \unit'_{|n|}(I, \Gamma) &\rar \gunit^*(n; T, \Xi; R, \Delta; I, \Gamma) / \unit(n; I, \Gamma)
\end{align*}
are bijective (they are always injective).
\item If \(n \geq 2\) and \(\Lambda\mathrm{sr}(\eta_1; I \rtimes R, \Gamma \rtimes \Delta) \leq n - 2\), then the relative maps
\begin{align*}
\punit_n(n; R, \Delta; I, \Gamma) / \unit'_{|n|}(I, \Gamma) &\rar \punit^*(n - 1; R, \Delta; I, \Gamma) / \unit(n - 1; I, \Gamma),\\
\gunit_n(n; T, \Xi; R, \Delta; I, \Gamma) / \unit'_{|n|}(I, \Gamma) &\rar \gunit^*(n - 1; T, \Xi; R, \Delta; I, \Gamma) / \unit(n - 1; I, \Gamma)
\end{align*}
are also bijective (they are injective for \(n \geq 2\) without restriction of the \(\Lambda\)-stable rank).
\end{enumerate}
\end{theorem}
\begin{proof}
The first statement follows from proposition \ref{HomStability} and the following observation: if \(f\), \(f'\) are two elements of \(\Hom_{m n}(m; S, \Theta; n; R, \Delta)\), \(f' = gf\), and \(g \in \unit(n; R, \Delta)\), then \(g \in \unit'_{|n|}\). In the general unitary case recall that \(\unit'_{|n|} \leqt \gunit_n(n; T, \Xi; R, \Delta)\) and \(\unit(n; R, \Delta) \leqt \gunit^*(n; T, \Xi; R, \Delta)\) are normal subgroups.

In \((2)\) note that this map is injective: indeed, if \(f\) and \(f'\) maps to the same element of \(\Hom^*(m - 1; S, \Theta; n - 1; R, \Delta)\), then there is \(g \in \unit(R_{|n|}, \Delta^{|n|}_{|n|}) \leq \unit'_{|n|}\) such that \(f\) and \(g f'\) induce the same morphism \((S_{|m|}, \Theta^{|m|}_{|m|}) \rar (R_{|n|}, \Delta^{|n|}_{|n|})\), hence \(f = g f'\). In the case of general linear group the proof is even simpler: if \(g \in \gunit_n(n; T, \Xi; R, \Delta)\) maps to an element of \(\unit(n - 1; R, \Delta)\), then the definition of the general unitary group shows that \(g\) actually lies in \(\unit'_{|n|}\).

It remains to prove the surjectivity in \((2)\). We will do it in the general unitary case, the other two cases are similar. Under the assumptions the map
\[\gunit_{n - 1}(n - 1; T, \Xi; R, \Delta) / \unit'_{|n - 1|} \rar \gunit^*(n - 1; T, \Xi; R, \Delta) / \unit(n - 1; R, \Delta)\]
is bijective. There is also an obvious map
\[\gunit_{n - 1}(n - 1; T, \Xi; R, \Delta) / \unit'_{|n - 1|} \rar \gunit_n(n; T, \Xi; R, \Delta) / \unit'_{|n|}\]
that maps \(g\) to \(\widetilde g\) with \(\alpha(\widetilde g) = \alpha(g) + e_{|n|, |n - 1|}\, \alpha_{|n - 1|}(g)\, e_{|n - 1|, |n|}\) and \(\gamma^\circ(\widetilde g) = \gamma^\circ(g) \dotplus \gamma^\circ_{|n - 1|}(g) \cdot e_{|n - 1|, |n|}\). This map makes the diagram commutative, hence the required surjectivity follows.

The relative versions follow from the absolute ones using the description of \(\punit(n; R, \Delta; I, \Gamma)\) from lemma \ref{ProjectiveUnitaryGroups} and a simple diagram chasing.
\end{proof}

\bibliographystyle{plain}  
\bibliography{references}

\end{document}